\documentclass[leqno,12pt]{amsart}
\usepackage{amssymb}
\usepackage{amsmath}
\usepackage{enumerate}
\usepackage{amsfonts}
\usepackage{hyperref}
\usepackage{soulutf8}
\usepackage{tikz}

\definecolor{darkgreen}{RGB}{45, 119, 75}

\newcommand{\supp}{\text{supp }}

\newtheorem{teo}{Theorem}[section]

\newtheorem{coro}[teo]{Corollary}
\newtheorem{lema}[teo]{Lemma}
\newtheorem{propo}[teo]{Proposition}
\newtheorem{remark}[teo]{Remark}
\newtheorem{defn}[teo]{Definition}

\numberwithin{equation}{section}

\hypersetup{
    colorlinks=true,
    linkcolor= blue,
    citecolor =cyan,
    urlcolor = teal,
}

\begin{document}

\title[Semigroups generated by even powers of~Dunkl operators]{On semigroups generated by sums of even powers of~Dunkl operators}

\author[ J. Dziuba\'nski and A. Hejna]{Jacek Dziuba\'nski and Agnieszka Hejna}


\begin{abstract}
On the Euclidean space $\mathbb R^N$ equipped with a normalized  root system $R$, a multiplicity function $k\geq 0$, and the associated measure $dw(\mathbf x)=\prod_{\alpha\in R} |\langle \mathbf x,\alpha\rangle|^{k(\alpha)}d\mathbf x$ we consider the differential-difference operator 
$$L=(-1)^{\ell+1} \sum_{j=1}^m T_{\zeta_j}^{2\ell},$$ where $\zeta_1,...,\zeta_m$ are  nonzero vectors in $\mathbb R^N$, which span $\mathbb R^N$, and $T_{\zeta_j}$ are the Dunkl operators.  The operator $L$ is essentially self-adjoint on $L^2(dw)$ and generates a semigroup $\{S_t\}_{t \geq 0}$ of linear self-adjoint contractions, which has the form $S_tf(\mathbf x)=f*q_t(\mathbf{x})$, $q_t(\mathbf x)=t^{-\mathbf N\slash (2\ell)}q(\mathbf x\slash t^{1\slash (2\ell)})$, where $q(\mathbf x)$ is the Dunkl transform of the function  $ \exp(-\sum_{j=1}^m \langle \zeta_j,\xi\rangle^{2\ell})$. We prove that $q(\mathbf x)$  satisfies the following exponential decay:
$$ |q(\mathbf x)| \lesssim  \exp(-c \| \mathbf x\|^{2\ell\slash (2\ell-1)})$$ 
for a certain constant $c>0$. Moreover, if $q(\mathbf x,\mathbf y)=\tau_{\mathbf x}q(-\mathbf y)$, then 
$|q(\mathbf x,\mathbf y)|\lesssim w(B(\mathbf x,1))^{-1} \exp(-c d(\mathbf x,\mathbf y)^{2\ell \slash (2\ell-1)})$, where $d(\mathbf x,\mathbf y)=\min_{\sigma\in G}\| \mathbf x- \sigma(\mathbf y)\| $, $G$~is the reflection group for $R$, and $\tau_{\mathbf x}$ denotes the Dunkl translation.  
\end{abstract}

\address{J. Dziuba\'nski and A. Hejna, Uniwersytet Wroc\l awski,
Instytut Matematyczny,
Pl. Grunwaldzki 2/4,
50-384 Wroc\l aw,
Poland}
\email{jdziuban@math.uni.wroc.pl}
\email{hejna@math.uni.wroc.pl}

\thanks{
Research supported by the National Science Centre, Poland (Narodowe Centrum Nauki), Grant 2017/25/B/ST1/00599.}

\maketitle
\section{Introduction}
Let   $\zeta_1,...,\zeta_m\in\mathbb R^N$ be non-zero vectors which span $\mathbb R^N$. For $\ell \in \mathbb{N}$ (which will be fixed throughout the paper) we consider the symmetric differential-difference operator $$ L=(-1)^{\ell +1} \sum_{j=1}^m T_{\zeta_j}^{2\ell},$$
where $T_{\zeta_j}$ are Dunkl operators associated with a normalized system of roots $R$ and a multiplicity function $k\geq 0$ (see Section~\ref{sec:preliminaries} for details). {Let $dw$ denote the related measure (see \eqref{eq:measure}).} The operator $L$ is essentially self-adjoint on $L^2(dw)$ and its closure generates a semigroup of self-adjoint linear contractions $\{S_t\}_{t \geq 0}$ on $L^2(dw)$. The semigroup has the form
\begin{equation}\label{eq:action_semigroup}
 S_tf(\mathbf x)=f*q_t(\mathbf x),
 \end{equation}
where $q_t(\mathbf x)=\mathcal F^{-1}(\exp(-t\sum_{j=1}^m \langle \zeta_j,\cdot \rangle^{2\ell} )(\mathbf x).$ {Here and subsequently, $*$ denotes the Dunkl convolution, while $\mathcal F$ and $\mathcal F^{-1}$ stand for the Dunkl transform and its inverse respectively (see~\eqref{eq:DunklTransform})}. Clearly, $ q_t \in \mathcal S(\mathbb R^N)$, and if we set $q(\mathbf x)=q_1(\mathbf x)$, then, by homogeneity,  
\begin{equation}\label{eq:appropriate_scaling}
    q_t(\mathbf x)=t^{-\mathbf N\slash (2\ell) } q\Big(\frac{\mathbf x}{ t^{1\slash (2\ell)}}\Big).
\end{equation}
Our first result is to prove that the decay of $q(\mathbf x)$  is exponential. This is stated in the following theorem. 
\begin{teo}\label{teo:main_no_translation}
There are  constants $C,c>0$ such that for all $\mathbf{x} \in \mathbb{R}^N$ we have 
$$ |q(\mathbf x)|\leq C\exp(-c\| \mathbf x\|^{2\ell\slash (2\ell -1)}).$$
\end{teo}
Let  $\tau_{\mathbf{x}}$ denote the Dunkl translation (see \eqref{eq:translation}). Then  $q_t(\mathbf x,\mathbf y)=\tau_{\mathbf x}q_t(-\mathbf y)$ are  the integral kernels of the operators $S_t$ with respect to the measure $dw$, that is, 
$$ S_tf(\mathbf x)=\int_{\mathbb{R}^N} q_t(\mathbf x,\mathbf y)f(\mathbf y)\, dw(\mathbf y).$$ 
Let 
$$d(\mathbf x,\mathbf y)=\min_{\sigma\in G}\| \sigma(\mathbf x)-\mathbf y\|$$
be the distance of the orbit of $\mathbf x$ to the orbit of $\mathbf y$, where $G$ denotes the  Weyl group associated with $R$ (see Section~\ref{sec:preliminaries}). We denote by  $B(\mathbf{x},r)$ the (closed) Euclidean ball centered at $\mathbf{x} \in \mathbb{R}^N$ and radius $r$. 
Our second result expresses the decay of $q_t(\mathbf x,\mathbf y)$ by means of the distance $d(\mathbf x,\mathbf y)$.
\begin{teo}\label{teo:theo2}
There are  constants $C,c>0$ such that for all $\mathbf{x},\mathbf{y} \in \mathbb{R}^{N}$ we have
\begin{equation}\label{eq:main}
   |q(\mathbf x,\mathbf y)|\leq C(\max\{w(B(\mathbf{x},1)),w(B(\mathbf{y},1)\})^{-1} \exp(-cd(\mathbf x,\mathbf y)^{2\ell\slash (2\ell -1)}). 
\end{equation}
\end{teo}
\begin{remark}
\normalfont
By a scaling argument applied to~\eqref{eq:main} (see~\eqref{eq:appropriate_scaling} and~\eqref{eq:homo}) we obtain that there are $C,c>0$ such that for all $\mathbf{x},\mathbf{y} \in \mathbb{R}^N$ and $t>0$ we have
\begin{align*}
    |q_{t}(\mathbf{x},\mathbf{y})| \leq C(\max\{w(B(\mathbf{x},t^{1/(2\ell)})),w(B(\mathbf{y},t^{1/(2\ell)})\})^{-1} \exp\left(-c\frac{d(\mathbf x,\mathbf y)^{2\ell\slash (2\ell -1)}}{t^{1\slash (2\ell-1)} }\right). 
\end{align*}
\end{remark}
To prove the first theorem we borrow ideas of~\cite{DHul} and~\cite{DHZ}. We first introduce a~family of weighted $L^2$-spaces with weights of exponential growth and prove that \eqref{eq:action_semigroup}  defines strongly continuous semigroups of linear operators on these spaces. This is done by proving G{\aa}rding inequalities for associated weighted linear forms and applying a theorem of J.-L. Lions (see Theorem~\ref{teo:Lions}).    We expect that if a~convolution operator preserves  weighted $L^2$-spaces with weights of exponential growth and has some smoothness properties, then its convolution  kernel should have some fast decay, and in fact it has. 

Let us note that the function $q(\mathbf x)$ is not radial. Therefore in the proof of Theorem \ref{teo:theo2} we cannot apply the formula of R\"osler (see~\eqref{eq:translation-radial}) for translations of radial functions. 
In order to prove Theorem \ref{teo:theo2} we use  methods developed in~\cite{DzHej2} based on the description of the support the Dunkl translations of compactly supported functions combined with the observation that any sufficiently regular fast decaying function can be written as a convolution of two functions such that one of them is radial (see \cite{DzHej2}).  Let us emphasis difficulties we have to face  when we apply the method of exponential weights. The first one is that the Dunkl operators do not satisfy the Leibniz rule. The second one concerns the lack of knowledge about boundedness of the Dunkl translations on $L^p(dw)$ spaces  and the fact that the translations do not form a group of operators as it is in the case of Lie groups.

{\section{Preliminaries and notation}\label{sec:preliminaries}

The Dunkl theory is a generalization of the Euclidean Fourier analysis. It started with the seminal article \cite{Dunkl} and developed extensively afterwards (see e.g. \cite{RoeslerDeJeu}, \cite{Dunkl0}, \cite{Dunkl3}, \cite{Dunkl2}, \cite{GR}, \cite{Roesler2}, \cite{Roesle99}, \cite{Roesler2003}, \cite{ThangaveluXu}, \cite{Trimeche2002}). 
In this section we present basic facts concerning the theory of the Dunkl operators.  For details we refer the reader to~\cite{Dunkl},~\cite{Roesler3}, and~\cite{Roesler-Voit}. 

We consider the Euclidean space $\mathbb R^N$ with the scalar product $\langle\mathbf x,\mathbf y\rangle=\sum_{j=1}^N x_jy_j
$, $\mathbf x=(x_1,...,x_N)$, $\mathbf y=(y_1,...,y_N)$, and the norm $\| \mathbf x\|^2=\langle \mathbf x,\mathbf x\rangle$. For a nonzero vector $\alpha\in\mathbb R^N$,  the reflection $\sigma_\alpha$ with respect to the hyperplane $\alpha^\perp$ orthogonal to $\alpha$ is given by
\begin{equation}\label{eq:reflection}
\sigma_\alpha (\mathbf x)=\mathbf x-2\frac{\langle \mathbf x,\alpha\rangle}{\| \alpha\| ^2}\alpha.
\end{equation}
In this paper we fix a normalized root system in $\mathbb R^N$, that is, a finite set  $R\subset \mathbb R^N\setminus\{0\}$ such that   $\sigma_\alpha (R)=R$ and $\|\alpha\|=\sqrt{2}$ for every $\alpha\in R$. The finite group $G$ generated by the reflections $\sigma_\alpha \in R$ is called the {\it Weyl group} ({\it reflection group}) of the root system. A~{\textit{multiplicity function}} is a $G$-invariant function $k:R\to\mathbb C$ which will be fixed and $\geq 0$  throughout this paper.

 Let
\begin{equation}\label{eq:measure}
dw(\mathbf x)=\prod_{\alpha\in R}|\langle \mathbf x,\alpha\rangle|^{k(\alpha)}\, d\mathbf x
\end{equation} 
be  the associated measure in $\mathbb R^N$, where, here and subsequently, $d\mathbf x$ stands for the Lebesgue measure in $\mathbb R^N$.
We denote by $\mathbf N=N+\sum_{\alpha \in R} k(\alpha)$ the homogeneous dimension of the system. Clearly, 
\begin{equation}\label{eq:homo} w(B(t\mathbf x, tr))=t^{\mathbf N}w(B(\mathbf x,r)) \ \ \text{\rm for all } \mathbf x\in\mathbb R^N, \ t,r>0 
\end{equation}
and
\begin{equation}\label{eq:change_var}
\int_{\mathbb R^N} f(\mathbf x)\, dw(\mathbf x)=\int_{\mathbb R^N} t^{-\mathbf N} f(\mathbf x\slash t)\, dw(\mathbf x)\ \ \text{for} \ f\in L^1(dw)  \   \text{\rm and} \  t>0.
\end{equation}
Observe that (\footnote{The symbol $\sim$ between two positive expressions means that their ratio remains between two positive constants.})
\begin{equation}\label{eq:behavior} w(B(\mathbf x,r))\sim r^{N}\prod_{\alpha \in R} (|\langle \mathbf x,\alpha\rangle |+r)^{k(\alpha)},
\end{equation}
so $dw(\mathbf x)$ is doubling, that is, there is a constant $C>0$ such that
\begin{equation}\label{eq:doubling} w(B(\mathbf x,2r))\leq C w(B(\mathbf x,r)) \ \ \text{ for all } \mathbf x\in\mathbb R^N, \ r>0.
\end{equation}

For $\xi \in \mathbb{R}^N$, the {\it Dunkl operators} $T_\xi$  are the following $k$-deformations of the directional derivatives $\partial_\xi$ by a  difference operator:
\begin{equation}\label{eq:T_j}
     T_\xi f(\mathbf x)= \partial_\xi f(\mathbf x) + \sum_{\alpha\in R} \frac{k(\alpha)}{2}\langle\alpha ,\xi\rangle\frac{f(\mathbf x)-f(\sigma_\alpha \mathbf x)}{\langle \alpha,\mathbf x\rangle}.
\end{equation}
The Dunkl operators $T_{\xi}$, which were introduced in~\cite{Dunkl}, commute and are skew-symmetric with respect to the $G$-invariant measure $dw$.
 Suppose that $\xi\ne 0$, $f,g \in C^1(\mathbb{R}^N)$ and $g$ is radial. The following Leibniz rule can be confirmed by a direct calculation:
\begin{equation}\label{eq:Leibniz}
T_{\xi}(f g)=f(T_\xi g)+g(T_\xi f). 
\end{equation}
For fixed $\mathbf y\in\mathbb R^N$ the {\it Dunkl kernel} $E(\mathbf x,\mathbf y)$ is the unique analytic solution to the system
$$ T_\xi f=\langle \xi,\mathbf y\rangle f, \ \ f(0)=1.$$
The function $E(\mathbf x ,\mathbf y)$, which generalizes the exponential  function $e^{\langle \mathbf x,\mathbf y\rangle}$, has the unique extension to a holomorphic function on $\mathbb C^N\times \mathbb C^N$. Let $\{e_j\}_{1 \leq j \leq N}$ denote the canonical orthonormal basis in $\mathbb R^N$ and let $T_j=T_{e_j}$. For multi-index $\beta=(\beta_1,\beta_2,\ldots,\beta_N)  \in\mathbb N_0^N$, we set 
$$ |\beta|=\beta_1+\beta_2 +\ldots +\beta_N,$$
$$\partial^{\beta}=\partial_1^{\beta_1} \circ \partial_2^{\beta_2}\circ \ldots \circ \partial_N^{\beta_N},$$
$$ T^\beta = T_1^{\beta_1}\circ T_2^{\beta_2}\circ \ldots \circ T_N^{\beta_N}.$$
In our further consideration we shall need the following lemma.
\begin{lema}\label{lem:Roesler_Dunkl_kernel}
For all $\mathbf{x} \in \mathbb{R}^N$, $\mathbf{z} \in \mathbb{C}^N$ and $\nu \in \mathbb{N}_0^{N}$ we have
$$|\partial^{\nu}_{\mathbf{z}}E(\mathbf{x},\mathbf{z})| \leq \|\mathbf{x}\|^{|\nu|}\exp(\|\mathbf{x}\|\|{\rm Re \;}\mathbf{z}\|).$$
In particular, 
\begin{equation} \label{eq:E} | E(i\xi, \mathbf x)|\leq 1 \quad \text{ for all } \xi,\mathbf x\in \mathbb R^N.
\end{equation}
\end{lema}
\begin{proof}
See~\cite[Corollary 5.3]{Roesle99}.
\end{proof}

The Dunkl transform
  \begin{equation}\label{eq:DunklTransform}\mathcal F f(\xi)=c_k^{-1}\int_{\mathbb R^N} E(-i\xi, \mathbf x)f(\mathbf x)\, dw(\mathbf x),
  \end{equation}
  where
  $$c_k=\int_{\mathbb{R}^N}e^{-\frac{\|\mathbf{x}\|^2}{2}}\,dw(\mathbf{x})>0,$$
   originally defined for $f\in L^1(dw)$, is an isometry on $L^2(dw)$, i.e.,
   \begin{equation}\label{eq:Plancherel}
       \|f\|_{L^2(dw)}=\|\mathcal{F}f\|_{L^2(dw)} \text{ for all }f \in L^2(dw),
   \end{equation} and preserves the Schwartz class of functions $\mathcal S(\mathbb R^N)$ (see \cite{deJeu}). Its inverse $\mathcal F^{-1}$ has the form
  \begin{equation}\label{eq:inverse} \mathcal F^{-1} g(x)=c_k^{-1}\int_{\mathbb R^N} E(i\xi, \mathbf x)g(\xi)\, dw(\xi).
  \end{equation}
  Obviously, for all $f \in \mathcal{S}(\mathbb{R}^N)$,  we have 
  \begin{equation}\label{eq:transform_deriv} 
       \mathcal F(T_{\zeta}f)(\xi)= - i \langle \zeta, \xi\rangle \mathcal Ff(\xi)  \text{ for all } \xi,\zeta \in \mathbb{R}^N,
  \end{equation}
  and, consequently,
  \begin{equation}\label{eq:transform_L}  \mathcal F(Lf)(\xi)=-\Big(\sum_{j=1}^m \langle \zeta_j, \xi\rangle^{2\ell}\Big)\mathcal Ff(\xi) \text{ for all }\xi \in \mathbb{R}^N.
  \end{equation}
The Dunkl transform $\mathcal F$ is an analogue of the classical Fourier transform.

The {\it Dunkl translation\/} $\tau_{\mathbf{x}}f$ of a function $f\in\mathcal{S}(\mathbb{R}^N)$ by $\mathbf{x}\in\mathbb{R}^N$ is defined by
\begin{equation}\label{eq:translation}
\tau_{\mathbf{x}} f(\mathbf{y})=c_k^{-1} \int_{\mathbb{R}^N}{E}(i\xi,\mathbf{x})\,{E}(i\xi,\mathbf{y})\,\mathcal{F}f(\xi)\,{dw}(\xi).
\end{equation}
  It is a contraction on $L^2(dw)$, however it is an open  problem  if the Dunkl translations are bounded operators on $L^p(dw)$ for $p\ne 2$.
  
  The following specific formula was obtained by R\"osler \cite{Roesler2003} for the Dunkl translations of (reasonable) radial functions $f({\mathbf{x}})=\tilde{f}({\|\mathbf{x}\|})$:
\begin{equation}\label{eq:translation-radial}
\tau_{\mathbf{x}}f(-\mathbf{y})=\int_{\mathbb{R}^N}{(\tilde{f}\circ A)}(\mathbf{x},\mathbf{y},\eta)\,d\mu_{\mathbf{x}}(\eta)\text{ for all }\mathbf{x},\mathbf{y}\in\mathbb{R}^N.
\end{equation}
Here
\begin{equation*}
A(\mathbf{x},\mathbf{y},\eta)=\sqrt{{\|}\mathbf{x}{\|}^2+{\|}\mathbf{y}{\|}^2-2\langle \mathbf{y},\eta\rangle}=\sqrt{{\|}\mathbf{x}{\|}^2-{\|}\eta{\|}^2+{\|}\mathbf{y}-\eta{\|}^2}
\end{equation*}
and $\mu_{\mathbf{x}}$ is a probability measure, 
which is supported in the set $\operatorname{conv}\mathcal{O}(\mathbf{x})$,  where $\mathcal O(\mathbf x) =\{\sigma(\mathbf x): \sigma \in G\}$ is the orbit of $\mathbf x$. Formula~\eqref{eq:translation-radial} implies that for all radial $f \in L^1(dw)$ and $\mathbf{x} \in \mathbb{R}^N$ we have
\begin{equation}\label{eq:translation-bounded}
    \|\tau_{\mathbf{x}}f(\mathbf{y})\|_{L^1(dw(\mathbf{y}))} \leq \|f(\mathbf{y})\|_{L^1(dw(\mathbf{y}))}.
\end{equation}

  {The \textit{Dunkl convolution\/} $f*g$ of two reasonable functions (for instance Schwartz functions) is defined by
$$
(f*g)(\mathbf{x})=c_k\,\mathcal{F}^{-1}[(\mathcal{F}f)(\mathcal{F}g)](\mathbf{x})=\int_{\mathbb{R}^N}(\mathcal{F}f)(\xi)\,(\mathcal{F}g)(\xi)\,E(\mathbf{x},i\xi)\,dw(\xi) \text{ for }\mathbf{x}\in\mathbb{R}^N
$$
or, equivalently, by}
\begin{align*}
  {(f{*}g)(\mathbf{x})=\int_{\mathbb{R}^N}f(\mathbf{y})\,\tau_{\mathbf{x}}g(-\mathbf{y})\,{dw}(\mathbf{y})=\int_{\mathbb R^N} f(\mathbf y)g(\mathbf x,\mathbf y) \,dw(\mathbf{y}) \text{ for all } \mathbf{x}\in\mathbb{R}^N},  
\end{align*}
where, here and subsequently, $g(\mathbf x,\mathbf y)=\tau_{\mathbf x}g(-\mathbf y)$. 
  
The {\it Dunkl Laplacian} associated with $R$ and $k$  is the differential-difference operator $\Delta=\sum_{j=1}^N T_{j}^2$, which  acts on $C^2(\mathbb{R}^N)$-functions by

\begin{equation}\label{eq:laplace_formula}
    \Delta f(\mathbf x)=\Delta_{\rm eucl} f(\mathbf x)+\sum_{\alpha\in R} k(\alpha) \delta_\alpha f(\mathbf x),
\end{equation}
$$\delta_\alpha f(\mathbf x)=\frac{\partial_\alpha f(\mathbf x)}{\langle \alpha , \mathbf x\rangle} - \frac{\|\alpha\|^2}{2} \frac{f(\mathbf x)-f(\sigma_\alpha \mathbf x)}{\langle \alpha, \mathbf x\rangle^2}.$$
Obviously, $\mathcal F(\Delta f)(\xi)=-\| \xi\|^2\mathcal Ff(\xi)$. The operator $\Delta$ is essentially self-adjoint on $L^2(dw)$ (see for instance \cite[Theorem\;3.1]{AH}) and generates the semigroup $e^{t\Delta}$  of linear self-adjoint contractions on $L^2(dw)$. The semigroup has the form
  \begin{equation}\label{eq:heat_semigroup}
  e^{t\Delta} f(\mathbf x)=\mathcal F^{-1}(e^{-t\|\xi\|^2}\mathcal Ff(\xi))(\mathbf x)=\int_{\mathbb R^N} h_t(\mathbf x,\mathbf y)f(\mathbf y)\, dw(\mathbf y),
  \end{equation}
  where the heat kernel 
  \begin{equation}\label{eq:heat_kernel}
      h_t(\mathbf x,\mathbf y)=\tau_{\mathbf x}h_t(-\mathbf y), \ \ h_t(\mathbf x)=\mathcal F^{-1} (e^{-t\|\xi\|^2})(\mathbf x)=c_k^{-1} (2t)^{-\mathbf N\slash 2}e^{-\| \mathbf x\|^2\slash (4t)},
  \end{equation}
  is a $C^\infty$-function of all variables $\mathbf x,\mathbf y \in \mathbb{R}^N$, $t>0$ and satisfies \begin{equation}\label{eq:symmetry} 0<h_t(\mathbf x,\mathbf y)=h_t(\mathbf y,\mathbf x),
  \end{equation}
 \begin{equation}\label{eq:heat_one} \int_{\mathbb R^N} h_t(\mathbf x,\mathbf y)\, dw(\mathbf y)=1.
 \end{equation}
  Set
$$V(\mathbf x,\mathbf y,t)=\max (w(B(\mathbf x,t)),w(B(\mathbf y, t))).$$ 
The following theorem was proved in~\cite[Theorem 4.1]{ADH}.
\begin{teo}\label{teo:heat}
There are constants $C,c>0$ such that for all $\mathbf{x},\mathbf{y} \in \mathbb{R}^N$ and $t>0$ we have
\begin{equation}\label{eq:Gauss}
h_t(\mathbf{x},\mathbf{y}) \leq C\,V(\mathbf{x},\mathbf{y},\!\sqrt{t\,})^{-1}\,e^{-\hspace{.25mm}c\hspace{.5mm}d(\mathbf{x},\mathbf{y})^2\slash t}.
\end{equation}
\end{teo}
 }

\section{Weighted Hilbert spaces and bilinear forms}
\subsection{Definition and properties of exponential weight functions}
For any $s>0$ and $\mathbf{x} \in \mathbb{R}^N$ let us define 
\begin{equation}\label{eq:eta}
\eta (\mathbf{x})=\exp(\sqrt{1+\|\mathbf{x}\|^2}), \ \ \eta(\mathbf x,s)=\exp( \sqrt{1+\|s\mathbf{x}\|^2}).
\end{equation}
Clearly, 
\begin{equation}\label{eq:eta_comp}
    e^{s\| x\|}\leq \eta(\mathbf x,s)\leq e^{s\| x\|+1}.
\end{equation}
\begin{lema}\label{lem:exp_derivatives}
For every $\beta \in \mathbb{N}_0^{N}$ there is a constant $C_\beta>0$ such that for all $\mathbf{x} \in \mathbb{R}^N$ and $s > 0$ we have
\begin{equation}\label{eq:exp_derivatives}
|\partial_{\mathbf{x}}^{\beta}\eta (\mathbf{x},s)| \leq C_{\beta}s^{|\beta|}\eta (\mathbf{x},s),
\end{equation}
where, here and subsequently, $\partial_{\mathbf{x}}^{\beta}$ denotes the partial derivative with respect to the variable $\mathbf{x}$.
\end{lema}
\begin{proof}
The proof is straightforward.
\end{proof}
\begin{lema}\label{lem:difference_part}
Suppose that $\phi:\mathbb{R}^N \times (0,\infty)\to \mathbb{R}$ is a $C^{\infty}(\mathbb{R}^N)$-function such that for any $\beta \in \mathbb{N}_{0}^{N}$ there is $C_{\beta}>0$ such that
\begin{equation}\label{eq:phi_assumption}
|\partial_{\mathbf{x}}^{\beta}\phi(\mathbf{x},s)| \leq C_{\beta}s^{|\beta|}\eta(\mathbf{x},s)  \text{ for all }\mathbf{x}\in \mathbb{R}^N\text{ and }s>1/4.
\end{equation} 
Then for every $\zeta \neq 0$  and every $\alpha\in R$ the functions  $$s^{-1}T_{\zeta} \phi(\mathbf{x},s) \ \ \text{\rm and } \ \ \psi_{\alpha}(\mathbf{x},s)=s^{-1}\frac{\phi(\mathbf{x},s)-\phi(\sigma_{\alpha}(\mathbf{x}),s)}{\langle \mathbf{x},\alpha \rangle}$$   satisfy~\eqref{eq:phi_assumption}.
\end{lema}

\begin{proof}
Thanks to Lemma~\ref{lem:exp_derivatives} and~\eqref{eq:T_j} it is enough to check the claim for $\psi_{\alpha}$ for all $\alpha \in R$.  Note that
\begin{align*}
\frac{\phi(\mathbf{x},s)-\phi(\sigma_{\alpha}(\mathbf{x}),s)}{s\langle \mathbf{x},\alpha \rangle}
&=-\langle \mathbf{x},\alpha \rangle^{-1}s^{-1}\int_0^1 \frac{d}{dt}\phi\Big(\mathbf{x}-2t\frac{\langle \mathbf{x},\alpha \rangle}{\|\alpha\|^2}\alpha,s\Big)\,dt\\&=c_{\alpha}s^{-1}\int_0^1 \Big\langle \big(\nabla_{\mathbf{x}}\phi \big)\Big(\mathbf{x}-2t\frac{\langle \mathbf{x},\alpha \rangle}{\|\alpha\|^{2}}\alpha,s\Big),\alpha\Big\rangle\,dt,
\end{align*}
therefore
\begin{align*}
\partial^{\beta}\Big\{\frac{\phi(\mathbf{x},s)-\phi(\sigma_{\alpha}(\mathbf{x}),s)}{s\langle \mathbf{x},\alpha \rangle}\Big\}=c_{\alpha}s^{-1}\int_0^1 \Big\langle \partial^{\beta}\Big\{\big(\nabla_{\mathbf{x}}\phi\big)\Big(\mathbf{x}-2t\frac{\langle \mathbf{x},\alpha \rangle}{\|\alpha\|^{2}}\alpha,s\Big)\Big\},\alpha\Big\rangle\,dt,
\end{align*}
so the claim is a consequence of~\eqref{eq:phi_assumption} for $\phi$.
\end{proof}

\begin{lema}\label{lem:derivative_formula}
Suppose that $C^{\infty} \ni \phi:\mathbb{R}^{N}\times (0,\infty)\to \mathbb{R}$ satisfies~\eqref{eq:phi_assumption}. Then for every $\zeta \neq 0$ there are $\phi_{\zeta},\phi_{\alpha,\zeta} \in C^{\infty}(\mathbb{R}^N\times (0,\infty))$, $ \alpha\in R$, which satisfy~\eqref{eq:phi_assumption}, such that for all $f \in C^1(\mathbb{R}^N)$, $\mathbf{x} \in \mathbb{R}^N$, and $s>1/4$ we have
\begin{equation}\begin{split}
T_\zeta & (f(\cdot)\phi(\cdot,s))(\mathbf x)=\phi(\mathbf x,s) T_\zeta f(\mathbf x)+f(\mathbf x)s\phi_{\zeta}(\mathbf x,s)+s\sum_{\alpha \in R} f(\sigma_{\alpha}(\mathbf x))\phi_{\alpha,\zeta}(\mathbf x,s).
\end{split}\end{equation}
\end{lema}

\begin{proof}
By~\eqref{eq:T_j} we have
\begin{equation}\label{eq:derivative_formula_1}
\begin{split}
&T_{\zeta}(f(\cdot)\phi(\cdot,s))(\mathbf{x})=\partial_{{\mathbf{x}\,,}\zeta} (f\phi(\cdot,s))(\mathbf{x})+\sum_{\alpha \in R}\frac{k(\alpha)}{2}\langle \alpha, \zeta \rangle \frac{f(\mathbf{x})\phi(\mathbf{x},s)-f(\sigma_{\alpha}(\mathbf{x}))\phi(\sigma_{\alpha}(\mathbf{x}),s)}{\langle \mathbf{x},\alpha\rangle}\\&=
f(\mathbf{x})\partial_{{\mathbf{x}\,,}\zeta}\phi(\mathbf{x},s)+\phi(\mathbf{x},s)\partial_{{\mathbf{x}\,,}\zeta} f(\mathbf{x})+\sum_{\alpha \in R}\frac{k(\alpha)}{2}\langle \alpha, \zeta\rangle \frac{f(\mathbf{x})\phi(\mathbf{x},s)-f(\sigma_{\alpha}(\mathbf{x}))\phi(\mathbf{x},s)}{\langle \mathbf{x},\alpha\rangle}\\&\quad +\sum_{\alpha \in R}\frac{k(\alpha)}{2}\langle \alpha,\zeta \rangle \frac{f(\sigma_{\alpha}(\mathbf{x}))\phi(\mathbf{x},s)-f(\sigma_{\alpha}(\mathbf{x}))\phi(\sigma_{\alpha}(\mathbf{x}),s)}{\langle \mathbf{x},\alpha\rangle}.\\
\end{split}
\end{equation}
Setting 
$$\phi_{\zeta}(\mathbf x,s)=s^{-1} \partial_{{\mathbf{x}\,,}\zeta} \phi(\mathbf x,s), \ \  \ \phi_{\alpha,\zeta}(\mathbf x,s)=s^{-1}\frac{k(\alpha)}{2}\langle \alpha,\zeta\rangle \frac{\phi(\mathbf x, s)-\phi (\sigma_\alpha(\mathbf x),s)}{\langle \mathbf x,\alpha \rangle}$$
and using  Lemma~\ref{lem:difference_part} we get the claim. 
\end{proof}

{
\begin{lema}\label{lem:derivative_formula_radial}
Suppose that  $C^{\infty} \ni \phi:\mathbb{R}^{N}\times (0,\infty)\to \mathbb{R}$ is a function such that $\phi (\mathbf x,s)=\phi(\mathbf x',s)$ for all $\|\mathbf x\|=\mathbf \|\mathbf{x}'\|$ and it satisfies~\eqref{eq:phi_assumption}. Then for every $\zeta\ne 0$  there is $\phi_\zeta\in C^{\infty}(\mathbb{R}^N\times (0,\infty))$ which satisfies~\eqref{eq:phi_assumption} such that for all $f \in C^1(\mathbb{R}^N)$, $\mathbf{x} \in \mathbb{R}^N$, and $s>1/4$ we have
\begin{equation*}
    T_\zeta(f(\cdot)\phi(\cdot,s))(\mathbf{x})=\phi(\mathbf{x},s)T_\zeta f(\mathbf{x})+s\phi_\zeta (\mathbf{x},s)f(\mathbf{x}).
\end{equation*}
\end{lema}

\begin{proof}
The claim follows directly by~\eqref{eq:Leibniz} and~\eqref{eq:exp_derivatives}.
\end{proof}}
For $\sigma \in G$ let $f^{\sigma}(\mathbf{x})=f(\sigma(\mathbf x))$. It is easy to check that  for all $\zeta \neq 0$ we have 
\begin{equation}\label{eq:derivative_to_reflected}
    T_{\zeta}f^{\sigma}(\mathbf{x})=(T_{\sigma(\zeta)}f)(\sigma(\mathbf{x}))\text{ for all }\mathbf{x}\in \mathbb{R}^N.
\end{equation}
Iteration of  Lemma~\ref{lem:derivative_formula} together with \eqref{eq:derivative_to_reflected} and  Lemma~\ref{lem:derivative_formula_radial} gives  the following proposition.
\begin{propo}\label{pro:beta}
For every $\beta \in \mathbb{N}_0^N$ there are functions $\phi_{\beta,\beta',\sigma}(\mathbf x,s)$ which satisfy~\eqref{eq:phi_assumption} such that  for all $f\in C_c^\infty (\mathbb R^N)$, $\mathbf x\in\mathbb R^N$, and $s>1/4$ we have 
\begin{equation}\label{eq:derivative_iteration}
\begin{split}
    T^\beta(f(\cdot)\eta(\cdot, s))(\mathbf x)&=T^\beta f(\mathbf x)\eta(\mathbf x,s)\\
   & +
   \sum_{\sigma\in G}\,  \sum_{|\beta'|< |\beta|}  s^{|\beta|-|\beta'|}(T^{\beta'}f)(\sigma(\mathbf x))\phi_{\beta, \beta', \sigma} (\mathbf x,s).
\end{split}\end{equation}\end{propo}
\subsection{Weighted Hilbert spaces}
We define a family $\{\mathcal H_s\}_{s>0}$  of weighted $L^2$-spaces by 
$$ \mathcal H_s=\Big\{f \in L^2(dw): \| f\|_{\mathcal H_s}^2:=\int_{\mathbb{R}^N} |f(\mathbf x)|^2\eta(\mathbf x,s)\, dw(\mathbf x)<\infty\Big\}. $$ 
To unify our notation we write
$$\mathcal{H}_0=L^2(dw).$$ 
Clearly, for $s_1 \leq s_2$ we have
\begin{equation}\label{eq:inclusion}
    \mathcal{H}_{s_2} \subset \mathcal{H}_{s_1} \text{ and }\|f\|_{\mathcal{H}_{s_1}} \leq \|f\|_{\mathcal{H}_{s_2}}.
\end{equation}
Let us note that for all $\mathbf{x} \in \mathbb{R}^N$ and $s>0$ we have
\begin{equation}
\eta(\mathbf x,2s)\leq \eta^{2}(\mathbf x,s)\leq e^2\eta(\mathbf x,2s).
\end{equation} 
Therefore, 
\begin{equation}\label{eq:comparison1}
 \| f\|^2_{\mathcal H_{2s}} \leq \int_{\mathbb{R}^N} |f(\mathbf x)|^2\eta^2(\mathbf x,s)\, dw(\mathbf x)\leq e^2 \| f\|_{\mathcal H_{2s}}^2.
 \end{equation}
Let us recall that $\eta(\mathbf{x})=\eta(\mathbf{x},1)$. The following corollary in a consequence of~\eqref{eq:derivative_iteration} and~\eqref{eq:comparison1}.
\begin{coro}
For every $\beta \in \mathbb{N}_0^{N}$ there is a constant $C_\beta>0$ such that for every $f\in C_c^\infty (\mathbb R^N)$ we have
\begin{equation}\label{eq:f_eta_1}
 \| T^\beta f\|_{L^2(\eta^2dw)}^2   \leq C_\beta \| T^\beta (f\eta)\|_{L^2(dw)}^2
  +C_\beta \sum_{|\beta'|<|\beta|} \| T^{\beta'} f\|_{L^2(\eta^2dw)}^2,
\end{equation}
\begin{equation}\label{eq:f_eta_0}
    \| T^\beta (f\eta)\|_{L^2(dw)}^2\leq C_\beta \| T^\beta f\|_{L^2(\eta^2dw)}^2 +C_\beta \sum_{|\beta'|<|\beta|} \| T^{\beta'} f\|_{L^2(\eta^2dw)}^2.
\end{equation}
\end{coro}

\begin{propo}\label{propo:pP}
For every $\delta>0$ and $\ell_1 \in \mathbb{N}$ (in particular, for $\ell_1=\ell$) there is a constant $C_{\delta,\ell_1}>0$ such that for all $f \in C^{\infty}_c(\mathbb{R}^N)$ we have 
\begin{equation}\label{eq:pP}
    \sum_{|\beta|<\ell_1} \| T^\beta f\|_{L^2(\eta^2dw)}^2 \leq \delta \sum_{j=1}^m \| T_{\zeta_j}^{\ell_1} f\|_{L^2(\eta^2dw)}^2+C_{\delta,\ell_1} \| f\|_{L^2(\eta^2dw)}^2.
\end{equation}
\end{propo}
\begin{proof}
 Thanks to ~\eqref{eq:Plancherel},~\eqref{eq:transform_deriv},  ~\eqref{eq:transform_L},  and  the fact that  $\zeta_1,...,\zeta_m$ span $\mathbb{R}^N$, we get that for every $\beta \in \mathbb{N}_0^N$ there is a constant $C_\beta>0$ such that 
\begin{equation}\label{eq:f_eta_2} 
\| T^\beta f\|_{L^2(dw)}^2\leq C_\beta \sum_{j=1}^m \| T_{\zeta_j}^{|\beta|} f\|_{L^2(dw)}^2.
\end{equation}
Moreover, for every $\ell_1 \in \mathbb{N}_0$, $\beta \in \mathbb{N}_0^N$ such that  $|\beta|<\ell_1$, and every $\delta>0$ there is a constant $C_{\beta,\delta}>0$ such that 
\begin{equation} \label{eq:f_eta_3}
\| T^\beta f\|_{L^2(dw)}^2\leq \delta \sum_{j=1}^m \| T_{\zeta_j}^{\ell_1} f\|_{L^2(dw)}^2 +C_{\beta,\delta} \| f\|^2_{L^2(dw)}.
\end{equation}
The proof of~\eqref{eq:pP} is by induction on $\ell_1$. Assume that~\eqref{eq:pP} holds for $\ell_1$. 
Using~\eqref{eq:f_eta_1} we have
\begin{equation}\label{eq:hyp_usage_0}
\begin{split}
    \sum_{|\beta|<\ell_1+1}\| T^\beta f\|_{L^2(\eta^2dw)}^2&\leq C\sum_{|\beta |<\ell_1+1} \| T^{\beta} (f\eta)\|_{L^2(dw)}^2 +C\sum_{|\beta'|<\ell_1} \|T^{\beta'} f\|^2_{L^2(\eta^2dw)}.
\end{split}
\end{equation}
Then, by~\eqref{eq:f_eta_3} (for the first summand) and induction hypothesis~\eqref{eq:pP} (for the second summand) for any $\varepsilon>0$ we get
\begin{equation}\label{eq:hyp_usage}
\begin{split}
     C\sum_{|\beta |<\ell_1+1} \| T^{\beta} (f\eta)\|_{L^2(dw)}^2 & +C\sum_{|\beta'|<\ell_1} \|T^{\beta'} f\|^2_{L^2(\eta^2dw)}\\
     &\leq \varepsilon C \sum_{j=1}^m \| T^{\ell_1+1}_{\zeta_j}(f\eta)\|_{L^2(dw)}^2+C'_\varepsilon \| f\eta\|_{L^2(dw)}^2\\
   & \ \ + \varepsilon C \sum_{j=1}^m \| T_{\zeta_j}^{\ell_1} f\|_{L^2(\eta^2dw)}^2+ C'C_\varepsilon \| f\|_{L^2(\eta^2dw)}^2.
   \end{split}
\end{equation}
Finally, joining~\eqref{eq:hyp_usage_0} and~\eqref{eq:hyp_usage} and applying~\eqref{eq:f_eta_0} we get
\begin{align*}
    \sum_{|\beta|<\ell_1+1}\| T^\beta f\|_{L^2(\eta^2dw)}^2 &\leq \varepsilon C_{\ell_1} \sum_{j=1}^m \| T_{\zeta_j}^{\ell_1+1} f\|^2_{L^2(\eta^2dw)}+\varepsilon C_{\ell_1}  \sum_{|\beta|<\ell_1 +1} \| T^\beta f\|^2_{L^2(\eta^2dw)}\\
   &\ \ +  \varepsilon C \sum_{j=1}^m \| T_{\zeta_j}^{\ell_1} f\|_{L^2(\eta^2dw)}^2+ C'C_\varepsilon \| f\|_{L^2(\eta^2dw)}^2.
\end{align*}
The proof is finished by taking $\varepsilon=\frac{1}{4}\min\{\delta,1\}(C_{\ell_1}+C)^{-1}$.
\end{proof}
\begin{propo}
Let $\beta \in \mathbb{N}_0^N$. There is a constant $C_{\beta}>0$ such that for all $f \in C^\infty_c(\mathbb{R}^N)$ we have
\begin{equation}\label{eq:pP_main}
    \|T^{\beta}f\|_{L^2(\eta^2 dw)}^2 \leq C_{\beta}\Big(\sum_{j=1}^m\|T^{|\beta|}_{\zeta_j}f\|^2_{L^2(\eta^2dw)}+\|f\|^2_{L^2(\eta^2dw)}\Big).
\end{equation}
\end{propo}

\begin{proof}
Thanks to~\eqref{eq:f_eta_1} and Proposition~\ref{propo:pP} with $\delta=1$ we get
\begin{align*}
    \|T^{\beta}f\|_{L^2(\eta^2 dw)}^2 \leq C_{\beta}\|T^{\beta}(f\eta)\|^2_{L^2(dw)}+C_{\beta}\sum_{j=1}^{m}\|T^{|\beta|}_{\zeta_j}f\|^2_{L^2(\eta^2dw)}+C_{\beta}\|f\|_{L^2(\eta^2dw)}^2.
\end{align*}
In order to estimate $\|T^{\beta}(f\eta)\|^2_{L^2(dw)}$, we use~\eqref{eq:f_eta_2}, then~\eqref{eq:f_eta_0}, which lead to
\begin{align*}
    \|T^{\beta}(f\eta)\|_{L^2(dw)}^2&\leq C_{\beta}'\sum_{j=1}^{m}\|T_{\zeta_j}^{|\beta|}(f\eta)\|^2_{L^2(dw)} \\&\leq C_{\beta}''\sum_{j=1}^{m}\|T_{\zeta_j}^{|\beta|}f\|^2_{L^2(\eta^2dw)}+C_{\beta}''\sum_{|\beta'|<|\beta|}\|T^{\beta'}f\|^2_{L^2(\eta^2dw)}.
\end{align*}
The claim follows by Proposition~\ref{propo:pP} with $\delta=1$ applied to $\sum_{|\beta'|<|\beta|}\|T^{\beta'}f\|^2_{L^2(\eta^2dw)}$.
\end{proof}
{
\begin{coro}\label{coro:garding_pre}
Let $n<\ell$ be a positive integer.  For every $\delta>0$ there is a constant $C=C_{\delta}>0$ such that for all $f \in C_c^\infty(\mathbb R^N)$ and for all $s>1/4$ we have
\begin{equation}\label{eq:garding_pre_1}
s^{2(\ell-n)}\sum_{|\beta|=n}\|T^{\beta}f\|^2_{L^2(\eta^2(\cdot,s)dw)} \leq \delta \sum_{j=1}^{m}\|T_{\zeta_j}^{\ell}f\|^2_{L^2(\eta^2(\cdot,s)dw)}+Cs^{2\ell}\|f\|^2_{L^2(\eta^2(\cdot,s)dw)},
\end{equation}
\begin{equation}\label{eq:garding_pre_2}
s^{2(\ell-n)}\sum_{|\beta|=n}\|T^{\beta}f\|^2_{\mathcal H_s} \leq \delta \sum_{j=1}^{m}\|T_{\zeta_j}^{\ell}f\|^2_{\mathcal H_s} +Cs^{2\ell}\|f\|^2_{\mathcal H_s}.
\end{equation}
\end{coro}

\begin{proof}
Let us apply~\eqref{eq:pP} to $f_{\{s\}}(\mathbf{x})=\frac{1}{s^{\mathbf{N}/2}}f(\mathbf{x}/s)$. Then~\eqref{eq:garding_pre_1} follows from the fact that
\begin{align*}
&\|T^{\beta}f_{\{s\}}\|^2_{L^2(\eta^2(\cdot)\,dw)}=s^{-2|\beta|}\|T^{\beta}f\|^2_{L^2(\eta^2 (\cdot , s)\,dw)}.
\end{align*}
Finally, \eqref{eq:garding_pre_2} is a consequence of~\eqref{eq:garding_pre_1} and~\eqref{eq:comparison1}. 
\end{proof}
\begin{coro}\label{coro:garding_pre_1}
There is a constant $C>0$ such that for all $f \in C^{\infty}_c(\mathbb{R}^N)$ and $s>1/4$ we have
\begin{equation}\label{eq:garding_pre_3}
    \sum_{|\beta|=\ell}\|T^{\beta}f\|^2_{\mathcal{H}_s} \leq C\sum_{j=1}^{m}\|T^\ell_{\zeta_j}f\|^2_{\mathcal{H}_s}+Cs^{2\ell}\|f\|^2_{\mathcal{H}_s}.
\end{equation}
\end{coro}

\begin{proof}
The proof is the same as the proof of Corollary~\ref{coro:garding_pre}, but instead of~\eqref{eq:pP} we use~\eqref{eq:pP_main}.
\end{proof}
\subsection{Weighted Sobolev spaces}\label{weight_sobol} For $s>0$ we define the weighted Sobolev space $V_{\ell,s}$ as the completion of $C_c^\infty(\mathbb R^N)$-functions in the norm 
$$ \| f\|^2_{V_{\ell,s}}=\| f\|^2_{\mathcal H_s}+\sum_{j=1}^m \| T_{\zeta_j}^\ell f\|_{\mathcal H_s}^2.$$
Clearly, $V_{\ell,s}\subset \mathcal H_s$. Moreover, $V_{\ell,s}$ is a dense subspace of $\mathcal H_s$. 

\begin{propo}\label{propo:definition_equivalence} Assume that  $f\in\mathcal H_s$. Then the following statements are equivalent: 

\begin{enumerate}[(a)]
    \item{$f \in V_{\ell,s}$; }\label{numitem:a_def}
    \item{for any $\beta \in \mathbb{N}_0^{N}$ such that $|\beta| \leq \ell$ there is a function $f_{\beta,s}\in \mathcal H_s$ such that for every $\varphi \in C^{\infty}_c(\mathbb{R}^N)$ we have 
\begin{equation}\label{eq:weak_der}
     (-1)^{|\beta|} \int_{\mathbb{R}^N} f(\mathbf x)T^\beta \varphi(\mathbf x)\, dw(\mathbf x)=\int_{\mathbb{R}^N} f_{\beta,s}(\mathbf x)\varphi(\mathbf x)\, dw(\mathbf x). 
     \end{equation} }\label{numitem:b_def}
\end{enumerate}
\end{propo}
\begin{proof}
See Appendix~\ref{sec:definition_equivalence}. 
\end{proof}
\begin{remark}If $0<s_1<s_2$ and $f \in V_{\ell,s_2}$, then $f \in V_{\ell,s_1}$ and the functions $f_{\beta,s_1}$ and $f_{\beta,s_2}$ from Proposition~\ref{propo:definition_equivalence} coincide. They will be denoted by $T^{\beta}f$.
\end{remark}

\subsection{Bilinear forms}
\begin{defn}\label{def:aaaa}\normalfont
For $s>1/4$ we define the bilinear form $a_{s}(\cdot,\cdot)$ with the domain $V_{\ell,s}$ by 
$$a_{s}(f,g)=- \sum_{j=1}^{m}\int_{\mathbb{R}^N} T_{\zeta_j}^{\ell}f(\mathbf{x})T_{\zeta_j}^{\ell}\big\{\overline{g}(\mathbf{x})\eta (\mathbf{x}, s)\big\}\,dw(\mathbf{x}).$$
\end{defn}

\begin{propo}\label{propo:form}
The form $a_{s}(f,g)$ is bounded on $V_{\ell,s}$. More precisely, there is a  constant $C>0$ such that   for every $s>1/4$ and every $f,g\in V_{\ell,s}$ we have 
\begin{equation}
    |a_{s}(f,g)|\leq C\Big(s^{2\ell} \| f\|_{\mathcal H_s}^2 + \sum_{j=1}^m \| T_{\zeta_j}^\ell  f\|_{\mathcal H_s}^2 \Big)^{1\slash 2}\Big(s^{2\ell}  \| g\|_{\mathcal H_s}^2 + \sum_{j=1}^m \| T_{\zeta_j}^\ell  g\|_{\mathcal H_s}^2 \Big)^{1\slash 2}.
\end{equation}
\end{propo}
\begin{proof}
{By Proposition \ref{pro:beta} there are functions $\phi_{j,\beta',\sigma}(\mathbf x,s)$, $\beta' \in \mathbb{N}_0^{N}$ and $\sigma \in G$,  such that 
\begin{align*}
    |\phi_{j,\beta',\sigma}(\mathbf x,s)|\leq C_{j,\beta',\sigma}\eta(\mathbf x,s) \text{ for all }\mathbf{x} \in \mathbb{R}^N,
\end{align*}
and 
\begin{equation*}
    \begin{split}
       |a_{s}(f,g)| & \leq  \Big|\sum_{j=1}^m \int_{\mathbb{R}^N} T_{\zeta_j}^\ell f(\mathbf x)T_{\zeta_j}^\ell \overline{g}(\mathbf x)\eta(\mathbf x,s)dw(\mathbf x)\Big|\\
       &\quad + \Big|\sum_{j=1}^m\sum_{\sigma\in G}\sum_{|\beta'|<\ell} \int_{\mathbb{R}^N} T_{\zeta_j}^\ell f(\mathbf x) s^{\ell-|\beta'|}(T^{\beta'}\overline{g})(\sigma(\mathbf x)) \phi_{j,\beta',\sigma}(\mathbf x,s)dw(\mathbf x) \Big|.
    \end{split}
\end{equation*}
Hence, using the Cauchy-Schwarz inequality we obtain 
$$ |a_{s}(f,g)| \leq     \sum_{j=1}^m \|T_{\zeta_j}^{\ell} f\|_{\mathcal H_s}\|T_{\zeta_j}^{\ell} g\|_{\mathcal H_s}+ C \sum_{j=1}^m \sum_{|\beta'|<\ell } \| T_{\zeta_j}^{\ell}  f\|_{\mathcal H_s} s^{\ell-|\beta'|} \| T^{\beta'} g\|_{\mathcal H_s}, $$
Now, applying  ~\eqref{eq:garding_pre_2} we get 
\begin{equation}\begin{split}\label{eq:bbound}
       |a_{s}(f,g)| &  
       \leq  \sum_{j=1}^m \|T_{\zeta_j}^{\ell} f\|_{\mathcal H_s}\|T_{\zeta_j}^{\ell} g\|_{\mathcal H_s} + C \Big(\sum_{j=1}^m \| T_{\zeta_j}^{\ell}  f\|_{\mathcal H_s}\Big) \Big(\sum_{j=1}^m  \| T_{\zeta_j}^{\ell} g\|_{\mathcal H_s}+s^{\ell}\| g\|_{\mathcal H_s}\Big). 
    \end{split}
\end{equation}
The proposition is a direct consequence of \eqref{eq:bbound}. }
\end{proof}

\begin{propo}[G\aa rding inequality]\label{propo:Garding2}
There are constants $\alpha, C_{\alpha}>0$ such that for all $s>1/4$ and $f \in V_{\ell,s}$ we have
\begin{align*}
-{\rm Re}\, a_{s}(f,f)+C_{\alpha}s^{2\ell}\|f\|_{\mathcal H_s}^2 \geq \alpha \|f\|_{V_{\ell,s}}^2.
\end{align*}
\end{propo}
\begin{proof}
Similarly to the proof of Proposition \ref{propo:form}, applying Proposition \ref{pro:beta},  we have
\begin{equation}
\begin{split}\label{eq:form_below}
    -  & {\rm Re}\, a_{s}(f,f)\geq \sum_{j=1}^m  \int_{\mathbb{R}^N} T_{\zeta_j}^\ell f(\mathbf x)T_{\zeta_j}^\ell \overline{f}(\mathbf x)\eta(\mathbf x,s)dw(\mathbf x) \\
    &\quad - \Big|\sum_{j=1}^m\sum_{\sigma\in G}\sum_{|\beta|<\ell} \int_{\mathbb{R}^N} T_{\zeta_j}^\ell f(\mathbf x) s^{\ell-|\beta|}(T^{\beta}\overline{f})(\sigma(\mathbf x)) \phi_{j,\beta,\sigma}(\mathbf x,s)dw(\mathbf x) \Big|\\
&\geq     \sum_{j=1}^m \|T_{\zeta_j}^{\ell} f\|_{\mathcal H_s}^2 - C \sum_{j=1}^m \sum_{|\beta|<\ell } \| T_{\zeta_j}^{\ell}  f\|_{\mathcal H_s} s^{\ell-|\beta|} \| T^{\beta} f\|_{\mathcal H_s}.\\
\end{split}\end{equation}
Using \eqref{eq:garding_pre_2} and the Cauchy-Schwarz inequality for any $\delta>0$ there is a constant $C_\delta'>0$ such that for any $\varepsilon >0$ we have  
\begin{equation}\begin{split}\label{eq:form_blow1}
C \sum_{j=1}^m & \sum_{|\beta|<\ell } \| T_{\zeta_j}^{\ell}  f\|_{\mathcal H_s} s^{\ell-|\beta|} \| T^{\beta} f\|_{\mathcal H_s} \\
&\leq  C \Big(\sum_{j=1}^m \| T_{\zeta_j}^{\ell}  f\|_{\mathcal H_s}\Big) \Big(\delta \sum_{j=1}^m \| T_{\zeta_j}^{\ell}  f\|_{\mathcal H_s} + C'_\delta s^{\ell} \| f\|_{\mathcal H_s}\Big)\\
&\leq C\delta m\sum_{j=1}^m \| T_{\zeta_j}^\ell f\|^2_{\mathcal H_s}+CC_\delta' \Big(\sum_{j=1}^m \| T_{\zeta_j}^\ell f\|_{\mathcal H_s} s^{\ell} \| f\|_{\mathcal H_s}\Big)\\
&\leq  C\delta m\sum_{j=1}^m \| T_{\zeta_j}^\ell f\|^2_{\mathcal H_s} +CC_\delta' \Big(\sum_{j=1}^m \varepsilon \| T_{\zeta_j}^\ell f\|_{\mathcal H_s}^2+ \frac{ms^{2\ell}}{4\varepsilon} \| f\|^2_{\mathcal H_s}\Big).
\end{split}\end{equation}
Taking $\delta, \varepsilon >0$ small enough such that $ C\delta m+CC_\delta'\varepsilon <1\slash 2$  we conclude the proposition from \eqref{eq:form_below} and \eqref{eq:form_blow1}. 
\end{proof}

\section{Perturbations of the bilinear form}
For $\varepsilon \geq 0$  and $s>1/4$ we consider the following bilinear form 
$$ b_{s,\varepsilon}(f,g)=a_{s}(f,g) +\varepsilon \sum_{j=1}^N \int_{\mathbb{R}^N} T_{j}f(\mathbf x)T_{j} \{ \overline{g}(\cdot)\eta(\cdot ,s)\}(\mathbf x)\, dw(\mathbf x)$$
 with the domain $V_{\ell,s}$. Let us note that $b_{s,0}(f,g)=a_{s}(f,g)$. 

\begin{propo}\label{propo:form_perturbation}
For every $\varepsilon \geq 0$ and $s>1\slash 4$ the form $b_{s,\varepsilon}$ is bounded on $V_{\ell,s}$.
\end{propo}

\begin{proof} Thanks to \eqref{eq:garding_pre_2} and \eqref{eq:garding_pre_3} there is a constant $C>0$ such for all $s>1/4$ we have $$ \sum_{j=1}^N\| T_{j} f\|^2_{\mathcal H_s}\leq C(\| f\|_{V_{\ell,s}}^2+ s^{2\ell} \| f\|^2_{\mathcal H_s}).$$
Hence, using  Lemma~\ref{lem:derivative_formula_radial} and then either \eqref{eq:garding_pre_2} or \eqref{eq:garding_pre_3}, we obtain  
\begin{equation}\begin{split}\label{eq:sum2}
    \Big|&\sum_{j=1}^N \int_{\mathbb{R}^N} T_{j}f(\mathbf x)T_{j} \{ \overline{g}(\cdot)\eta(\cdot ,s)\}(\mathbf x)\, dw(\mathbf x)\Big| \\
    &\leq \sum_{j=1}^N\| T_{j}f\|_{\mathcal H_s}\| T_{j}g\|_{\mathcal H_s}
     +C \sum_{j=1}^N\| T_{j}f\|_{\mathcal H_s}s\| g\|_{\mathcal H_s}\\
     &\leq C\Big( \| f\|^2_{ V_{\ell,s}} +s^{2\ell} \| f\|^2_{\mathcal H_s}\Big)^{1\slash 2} \Big( \| g\|^2_{ V_{\ell,s}} +s^{2\ell} \| g\|^2_{\mathcal H_s}\Big)^{1\slash 2}. 
\end{split}\end{equation}
 Now Proposition \ref{propo:form_perturbation} follows from \eqref{eq:sum2} and  Proposition~\ref{propo:form}. 
\end{proof} 
\begin{propo}[G{\aa}rding inequality for the perturbed bilinear form]\label{propo:garding_perturbed}
There are $\varepsilon_0>0$ and $\alpha, C_\alpha>0$ such that for all $0\leq \varepsilon\leq \varepsilon_0$, $f \in V_{\ell,s}$, and every $s>1\slash 4$ we have
\begin{align*}
-{\rm Re}\, b_{s,\varepsilon}(f,f)+C_{\alpha}s^{2\ell}\|f\|_{\mathcal H_s}^2 \geq \alpha \|f\|_{V_{\ell,s}}^2.
\end{align*}
\end{propo}
\begin{proof} 
It suffices to take  $\varepsilon_0>0  $ small enough and apply  Proposition \ref{propo:Garding2} together with \eqref{eq:sum2}. 
\end{proof}

 The number $\varepsilon_0$ from Proposition \ref{propo:garding_perturbed}  will be fixed throughout the remaining part of the paper.
}

\section{Semigroups of operators and Lions theorem}
\subsection{Lions theorem}
The following theorem is essentially  due to J.-L. Lions~\cite{Lions}. Its proof, which includes holomorphy of the semigroup under consideration, and which is a combination of a number of propositions from~\cite{Lions} and~\cite{Pazy}, can be found in~\cite[Proposition (1.1)]{DHul}.

\begin{teo}\label{teo:Lions}
Let $\mathcal{H}$ be a Hilbert space and $V$ be a dense subspace of $\mathcal{H}$ such that $V$ is a Hilbert space with the inner product $\langle \cdot,\cdot \rangle_{V}$ and the norm $\|\cdot\|_{V}$, and for some constant $c>0$ we have $\|f\|_{\mathcal{H}} \leq c\|f\|_{V}$ for all $f \in V$. Let $b(\cdot,\cdot)$ be a bounded  bilinear form on $V$. It defines an operator $A:D(A) \mapsto \mathcal{H}$ as follows
\begin{align*}
D(A)=\{f \in V\,:\,|b(f,g)| \leq C_{g}\|f\|_{\mathcal{H}} \text{ for all }g\in V\}, \ \ \ \langle Af,g \rangle_{\mathcal{H}}=b(f,g).
\end{align*}
Suppose that for some $\alpha>0$ and $\lambda_0 \in \mathbb{R}$ we have
\begin{equation}\label{eq:Garding_general}
\alpha\|f\|_{V}^2 \leq -{\rm Re\;}b(f,f)+\lambda_0\|f\|_{\mathcal{H}}^2 \quad \text{ for all }f \in V.
\end{equation}
Then $A$ is the infinitesimal generator of a strongly continuous semigroup $\{T_{t}\}_{t \geq 0}$ of operators on $\mathcal{H}$ which is holomorphic in a sector 
$$S_{\kappa} =\{z \in \mathbb{C}: |{\rm Arg}\,z|<\kappa \}$$
for some $\kappa >0$. Moreover, 
\begin{equation}\label{eq:semigroup_growth_pre}
\|T_tf\|_{\mathcal{H}} \leq \exp(\lambda_0t)\|f\|_{\mathcal{H}} \text{ for all }t \geq 0 \text{ and }f \in \mathcal{H}.
\end{equation}
\end{teo}

\subsection{Semigroup \texorpdfstring{$\{S_t\}_{t \geq 0}$}{St} of operators on \texorpdfstring{$L^2(dw)$}{L2(dw)}}
{{For $\varepsilon \in \{0,\varepsilon_0\}$ let us define the symmetric bilinear form }
$$ b_{0,\varepsilon}(f,g)=- \sum_{j=1}^m \int_{\mathbb{R}^N} T^\ell_{\zeta_j} f(\mathbf x)T^\ell_{\zeta_j} \overline{g}(\mathbf x)\, dw(\mathbf x)+\varepsilon \sum_{j=1}^N \int_{\mathbb{R}^N}  T_{j}f(\mathbf x)T_{j} \overline{g}(\mathbf x)\, dw(\mathbf x)$$
with the domain $V_{\ell,0}=\{ f\in L^2(dw): (1+\| \xi\|)^\ell \mathcal Ff(\xi)\in L^2(dw)\}$ and the norm 
$$ \| f\|_{V_{\ell,0}}^2=\sum_{j=1}^m \| T^\ell_{\zeta_j} f\|_{\mathcal H_0}^2+\| f\|_{\mathcal H_0}^2.$$
The form can be written by means of the Dunkl transform  as  $$b_{0,\varepsilon} (f,g)= \int_{\mathbb{R}^N} \mathcal Ff(\xi)\mathcal F\overline{g}(\xi)\Big( -\sum_{j=1}^m \langle \zeta_j,\xi\rangle^{2\ell} +\varepsilon \| \xi\|^2\Big) dw(\xi). $$
\begin{propo}\label{propo:b_0}
Let $\varepsilon \in \{0,\varepsilon_0\}$. The form $b_{0,\varepsilon}$ is bounded on $V_{\ell,0}$. Moreover, it satisfies the following G\aa rding inequality: there are $\lambda_0,\alpha>0$ such that
\begin{equation}\label{eq:Gardin_V_0}
\alpha\|f\|_{V_{\ell,0}}^2 \leq -{\rm Re}\,b_{0,\varepsilon}(f,f)+\lambda_0\|f\|_{\mathcal{H}_0}^2 \quad \text{ for all }f \in V_{\ell,0}. 
\end{equation} 
\end{propo}

{\begin{proof}
By the Cauchy--Schwarz inequality and~\eqref{eq:Plancherel} we have
\begin{align*}
    |b_{0,\varepsilon}(f,g)| &\leq \sum_{j=1}^{m}\|T_{\zeta_j}^{\ell}f\|_{L^2(dw)}\|T_{\zeta_j}^{\ell}g\|_{L^2(dw)}+\varepsilon\sum_{j=1}^{N}\|T_jf\|_{L^2(dw)}\|T_jg\|_{L^2(dw)}\\& \leq C \sum_{j=1}^{m}\||\xi|^{\ell}\mathcal{F}f(\xi)\|_{L^2(dw(\xi))}\||\xi|^{\ell}\mathcal{F}g(\xi)\|_{L^2(dw(\xi))}\\
    &\quad +\varepsilon C \sum_{j=1}^{N}\||\xi|\mathcal{F}f(\xi)\|_{L^2(dw(\xi))}\||\xi|\mathcal{F}g(\xi)\|_{L^2(dw(\xi))}\\&\leq C \|(1+|\xi|)^{\ell}\mathcal{F}f(\xi)\|_{L^2(dw(\xi))}\|(1+|\xi|)^{\ell}\mathcal{F}g(\xi)\|_{L^2(dw(\xi))},
\end{align*}
which implies that the form $b_{0,\varepsilon}$ is bounded on $V_{\ell,0}$. The G\aa rding inequality can be verified by the same way.
\end{proof}}
{As the consequence of the boundedness of $b_{0,\varepsilon}$, we conclude that it defines a self-adjoint linear operator $A^{(\varepsilon)}$,  which, thanks to Theorem~\ref{teo:Lions} and the G\aa rding   inequality \eqref{eq:Gardin_V_0}, generates a strongly continuous semigroup $\{S_t^{(\varepsilon)}\}_{t \geq 0}$  of bounded self-adjoint linear operators on $\mathcal H_0=L^2(dw)$, which has the form 
$$ S_t^{(\varepsilon)}f(\mathbf x)= f*q^{(\varepsilon)}_t(\mathbf x), $$ 
where 
\begin{equation}\label{eq:q_tepsilon}
   q^{(\varepsilon)}_t(\mathbf x)=\mathcal{F}^{-1}\Big(\exp\Big(-t\Big(\sum_{j=1}^m \langle \zeta_j,\cdot \rangle ^{2\ell} -\varepsilon \| \cdot\|^2\big)\Big)\Big)(\mathbf  x). 
\end{equation}
Let us also remark (see Proposition~\ref{pro:core}) that the operator $A^{(\varepsilon)}$ is the closure in the space $\mathcal H_0$ of
\begin{equation}\label{eq:L}
    L^{(\varepsilon)}=\sum_{j=1}^{m}T^{2\ell}_{\zeta_j}-\varepsilon \Delta,
\end{equation}
 initially defined on  $C^{\infty}_c(\mathbb{R}^N)$} (for the proof see Appendix~\ref{sec:core} with $s=0$).

\subsection{Semigroups on weighted Hilbert spaces}\label{sec:semigroup}
We are in a position to apply Theorem \ref{teo:lions_app} to  the weighted bilinear forms $b_{ s,\varepsilon}$, where  $\varepsilon  \in \{0,\varepsilon_0\}$ and $s>1/4$. 
Let us remind that the forms $b_{s,\varepsilon}$ are bounded (see Propositions~\ref{propo:form} and~\ref{propo:form_perturbation}). Let  $A_s^{(\varepsilon)}$ be the operator associated with the form $b_{s,\varepsilon}$ with its domain $D(A_s^{(\varepsilon)})\subset V_{\ell,s}\subset \mathcal H_s$. 
The following theorem is a direct consequence of 
Propositions \ref{propo:form_perturbation}, \ref{propo:garding_perturbed}, and Theorem \ref{teo:Lions}. 
} 
\begin{teo}\label{teo:lions_app}
Let $\varepsilon \in \{0,\varepsilon_0\}$. There are constants $c_0,  \kappa  >0$ such that for all $s>1/4$ the operator $A_s^{(\varepsilon)}$ is the infinitesimal generator of a strongly continuous semigroup $\{S^{\{\varepsilon,s\}}_t\}_{t \geq 0}$ of operators on $\mathcal H_s$ which is holomorphic is a sector
$$\{z \in \mathbb{C}: |\text{Arg}\,z|<\kappa \},$$
 which for all $f \in \mathcal H_s$ satisfies
\begin{equation}\label{eq:semigroup_growth}
\|S^{\{\varepsilon,s\}}_t f\|_{\mathcal H_s} \leq \exp(c_0 s^{2\ell}t)\|f\|_{\mathcal H_s}
\end{equation}
for all $t \geq 0$ and for all $f \in \mathcal H_s$.
\end{teo}

{
  Clearly, $\mathcal H_{s_1}\subset \mathcal H_{s_2}\subset \mathcal H_{0}$ and $V_{\ell, s_1}\subset  V_{\ell, s_2}\subset V_{\ell, 0}$ for $s_1\geq s_2\geq 0$.
The next theorem asserts that the semigroups $\{S_t^{\{\varepsilon,s\}}\}_{t \geq0}$ can be thought as the semigroup $\{S^{(\varepsilon)}_t\}_{t \geq 0}$ acting on $\mathcal H_s$. 

\begin{teo}\label{teo:one_semigroup}
Let $\varepsilon \in \{0,\varepsilon_0\}$. For all $s>1/4$ and $f\in\mathcal H_s\subset L^2(dw)$ we have 
$$ S^{\{\varepsilon,s\}}_t f=S^{(\varepsilon)}_t f=f*q^{(\varepsilon)}_t \text{ for all }t \geq 0.$$
\end{teo}
\begin{proof}
See Appendix~\ref{sec:selfadjoint}. 
\end{proof}
\begin{propo}\label{pro:core}
Let $\varepsilon \in \{0,\varepsilon_0\}$. For  $s> 1/4$ let $\lambda>c_0s^{2\ell}$, where $c_0>0$ is the constant from \eqref{eq:semigroup_growth}.  Then for every $n\in \mathbb N$ the space $C_c^\infty(\mathbb R^N)$ is a core for $(\lambda I-A_s^{(\varepsilon)})^n$. 
\end{propo}
\begin{proof}
See Appendix~\ref{sec:core}.
\end{proof}
 
}

\section{Pointwise estimates for integral kernel of \texorpdfstring{$S_t$}{St}}
{
We define the sequence $\{d(n)\}_{n \in \mathbb{N}}$ inductively by
\begin{align*}
    \begin{cases}
    d(1)=2,\\
    d(n+1)=2d(n)+2 \text{ for }n \geq 2.
    \end{cases}
\end{align*}
\begin{lema}\label{lem:mixed_norms}
For every $\beta\in \mathbb N_0^N$ there is a constant $C>0$ such that for every $s>1/4$ and every $f\in C_c^\infty(\mathbb R^N) $ we have 
\begin{equation}\label{eq:mixed_1}
    \| T^\beta f\|_{\mathcal H_s}^2\leq C\Big( s^{d(|\beta|)} \|f\|_{\mathcal H_{2^{|\beta|}s}}^2 + \sum_{|\beta'|\leq |\beta|+1} \| T^{\beta'}f\|_{L^2(dw)}^2\Big). 
\end{equation}
\end{lema}

\begin{proof}
The proof goes by induction on $|\beta|$. First, let us note that $\eta(\mathbf{x}, s)=\eta(\mathbf{x}',s)$ for $\|\mathbf{x}\|=\|\mathbf{x}'\|$, so for any function $f \in C^{\infty}_c(\mathbb{R}^N)$  integration by parts (see~\eqref{eq:Leibniz}) gives
\begin{equation}\label{eq:parts}
\begin{split}
    &\int_{\mathbb{R}^N} T_jf(\mathbf{x})T_j\overline{f}(\mathbf{x})\eta(\mathbf x,s)\,dw(\mathbf{x})\\&=
    -\int_{\mathbb{R}^N}f(\mathbf{x})T^2_j\overline{f}(\mathbf{x})\eta (\mathbf{x}, s)\,dw(\mathbf{x})-\int_{\mathbb{R}^N}f(\mathbf{x})T_j\overline{f}(\mathbf{x})(\partial_{\mathbf{x},j}\eta)(\mathbf{x},s)\,dw(\mathbf{x}).
\end{split}
\end{equation}
The claim for $|\beta|=1$ follows from~\eqref{eq:parts}, the  Cauchy--Schwarz inequality,  and Lemma~\ref{lem:exp_derivatives}, because
\begin{align*}
    \left|\int_{\mathbb{R}^N}f(\mathbf{x})T^2_j\overline{f}(\mathbf{x})\eta (\mathbf{x},s)\,dw(\mathbf{x})\right| \leq \|T_j^2f\|_{L^2(dw)}^2+C\|f\|^2_{\mathcal H_{2s}},
\end{align*}
\begin{align*}
    \left|\int_{\mathbb{R}^N}f(\mathbf{x})T_j\overline{f}(\mathbf{x}) \partial_{\mathbf{x},j}\eta(\mathbf{x}, s)\,dw(\mathbf{x})\right| \leq \|T_jf\|_{L^2(dw)}^2+Cs^2\|f\|^2_{\mathcal H_{2s}}.
\end{align*}
Assume that \eqref{eq:mixed_1} is satisfied for $\beta \in \mathbb{N}_0^N$ such that $|\beta|=n$. Consider multi-index $\beta+e_j$, where $|\beta|=n$. Then by~\eqref{eq:mixed_1} with $f$ replaced by $T_jf$ we get 
\begin{equation}\label{eq:TTj}
    \begin{split}
        \| T^{\beta+e_j} f\|^2_{\mathcal H_s} 
        &\leq C s^{d(|\beta|)}\| T_jf\|^2_{\mathcal H_{2^{|\beta|}s}} +C\sum_{|\beta'| \leq |\beta|+1} \| T^{\beta'} T_jf\|_{L^2(dw)}^2\\&= C s^{d(|\beta|)}\| T_jf\|^2_{\mathcal H_{2^{|\beta|}s}} +C\sum_{|\beta'| \leq |\beta+e_j|+1} \| T^{\beta'}f\|_{L^2(dw)}^2.\\
    \end{split}
\end{equation}
Using again the Cauchy-Schwarz inequality together with Lemma~\ref{lem:exp_derivatives}, we obtain  
 \begin{align*}
    s^{d(|\beta|)}\left|\int_{\mathbb{R}^N}f(\mathbf{x})T^2_j\overline{f}(\mathbf{x})\eta (\mathbf{x},2^{|\beta|}s)\,dw(\mathbf{x})\right| \leq \|T_j^2f\|_{L^2(dw)}^2+C's^{2d(|\beta|)}\|f\|^2_{\mathcal H_{2^{|\beta|+1}s}} 
\end{align*}
and 
\begin{align*}
    s^{d(|\beta|)}\left|\int_{\mathbb{R}^N}f(\mathbf{x})T_j\overline{f}(\mathbf{x}) \partial_{\mathbf{x},j}\eta(\mathbf{x}, 2^{|\beta|}s)\,dw(\mathbf{x})\right| \leq \|T_jf\|_{L^2(dw)}^2+C's^{2d(|\beta|)+2}\|f\|^2_{\mathcal H_{2^{|\beta|+1}s}}.
\end{align*}
Hence, repeating the calculation presented in \eqref{eq:parts} we get 
    \begin{equation}\label{eq:Tj}
 s^{d(|\beta|)}\| T_jf\|^2_{\mathcal H_{2^{|\beta|}s}} \leq C'' s^{2d(|\beta|)+2} \| f\|^2_{\mathcal H_{2^{{|\beta|+1}}s}} +\| T^2_jf\|_{L^2(dw)}^2 + \| T_jf\|_{L^2(dw)}^2. 
    \end{equation}
 Now  \eqref{eq:Tj} together with \eqref{eq:TTj}    completes the proof.
\end{proof}

} 
\begin{lema}\label{lem:spectral}
Let  $\varepsilon \in \{0,\varepsilon_0\}$ and $\beta\in\mathbb N_0^N$. There are constants $C,\lambda_0>0$, and $M \in \mathbb{N}$ such that for all $\lambda>\lambda_0$ and $f \in C_c^\infty (\mathbb{R}^N)$, we have
\begin{align*}
    \|T^\beta f\|_{L^2(dw)} \leq C\|(\lambda I-A^{(\varepsilon)})^Mf\|_{L^2(dw)},
\end{align*}
\end{lema}

\begin{proof}
Let us recall that $(\lambda I-A^{(\varepsilon)})^Mf=(\lambda I-L^{(\varepsilon)})^Mf$ for $f \in C^{\infty}_c(\mathbb{R}^N)$. The lemma is a consequence of~\eqref{eq:Plancherel} and  \eqref{eq:transform_deriv}.
\end{proof}

Combination  of Lemma~\ref{lem:mixed_norms} and Lemma~\ref{lem:spectral} leads to the following corollary.
\begin{coro}\label{coro:higher_derivatives}
Let $\varepsilon \in \{0,\varepsilon_0\}$ and $\beta  \in \mathbb{N}_0^N$. There are constants $C,\lambda_0 >0$, and $M=M_{\beta} \in \mathbb{N}$ such that for all $f \in C_c^\infty (\mathbb{R}^N)$, $\lambda\geq \lambda_0$,  and $s>1/4$, we have
\begin{align*}
    \|T^\beta f\|_{\mathcal H_s}^2 \leq C\|(\lambda I-A^{(\varepsilon)})^{M}f\|_{L^2(dw)}^2+Cs^{d(|\beta|)}\|f\|_{\mathcal H_{2^{|\beta|}s}}^2.
\end{align*}
\end{coro}
{
\begin{lema}\label{lem:T_j(fe)inL2}
 Let $\varepsilon \in \{0,\varepsilon_0\}$ and $\beta \in \mathbb{N}_0^N$. There are  constants $C,\lambda_0>0$, and $M=M_\beta>0$ such that for all $s>1/4$, $\lambda\geq \lambda_0$, and $f\in C_c^\infty (\mathbb{R}^N)$  we have
\begin{align*}
\|T^{\beta}(f(\cdot)\eta(\cdot,s))\|_{L^2(dw)}^2 \leq Cs^{2|\beta|}\|(\lambda I-A^{(\varepsilon)})^{M}f\|_{L^2(dw)}^2+Cs^{2|\beta|+d(|\beta|)}\|f\|_{\mathcal H_{2^{|\beta|+1}s}}^2.
\end{align*}
\end{lema}

\begin{proof} 
By Proposition~\ref{pro:beta} we get
\begin{align*}
    \|T^{\beta}(f(\cdot)\eta(\cdot,s))\|_{L^2(dw)}^2 \leq C\sum_{|\beta'| \leq |\beta|}s^{2(|\beta|-|\beta'|)}\|T^{\beta'}f\|_{\mathcal{H}_{2s}}^2.
\end{align*}
Then, applying Corollary~\ref{coro:higher_derivatives} to each term of the sum we obtain the claim.
\end{proof}

\begin{lema}\label{lem:pointwise_pre}
Let $\varepsilon \in \{0,\varepsilon_0\}$ and $\beta \in \mathbb{N}_0^N$. There are constants $C,c>0$ such that for all $s>1/4$, $f \in \mathcal H_{2^{|\beta|+1}s}$, and $1/2<t<2$ we have
\begin{equation}
\|T^\beta ((S_t^{(\varepsilon)}f)(\cdot)\eta(\cdot ,s))\|_{L^2(dw)} \leq C\exp(cs^{2\ell})\|f\|_{\mathcal H_{2^{|\beta|+1}s}}.
\end{equation}
\end{lema}

\begin{proof} 
Let $M=M_{\beta}$ be as in Lemma~\ref{lem:T_j(fe)inL2} and   let $c_0$ be the constant from~\eqref{eq:semigroup_growth}. We claim that Lemma \ref{lem:T_j(fe)inL2} is satisfied
if $f\in D((A^{(\varepsilon)}_{s_1})^M)$, where $s_1=2^{|\beta|+1}s$, and $\lambda > \max (\lambda_0, c_0s_1^{2\ell})$. Indeed, since    $C_c^\infty(\mathbb{R}^N)$ is a core for $(\lambda I-A^{(\varepsilon)}_{s_1})^M$ (see Proposition \ref{pro:core}) and $D((A_{s_1}^{(\varepsilon)})^M )\subset   D((\lambda I-A^{(\varepsilon)}_{s_1})^M)$, there are $f_n\in C_c^\infty (\mathbb R^N)$   such that 
$$\lim_{n\to\infty} \|f_n-f\|_{\mathcal H_{s_1}}+\|(\lambda I-A^{(\varepsilon)}_{s_1})^Mf_n-(\lambda I-A^{(\varepsilon)}_{s_1})^Mf\|_{\mathcal H_{s_1}}=0. $$ 
Consequently, by~\eqref{eq:inclusion},~\eqref{eq:comparison1}, and Corollary \ref{coro:inclusions} in Appendix~\ref{sec:selfadjoint} we have
$$\lim_{n\to\infty} \|f_n\eta(\cdot, s)-f\eta(\cdot, s)\|_{L^2(dw)}+\|(\lambda I-A^{(\varepsilon)})^Mf_n-(\lambda I-A^{(\varepsilon)})^Mf\|_{L^2(dw)}=0. $$
Now the claim follows, because $T^\beta$ is  closed on $L^2(dw)$. 

Set $\lambda=\max (\lambda_0, 2 c_0s_1^{2\ell})$. If $f\in \mathcal H_{s_1}$, then $S^{(\varepsilon)}_tf\in D((A_{s_1}^{(\varepsilon)})^M)$, because $\{S^{(\varepsilon)}_t\}_{t \geq 0}$ is analytic. 
Hence, by Lemma~\ref{lem:T_j(fe)inL2},  we get
\begin{align*}
    \|T^{\beta} ((S_t^{(\varepsilon)}f)(\cdot)\eta(\cdot, s))\|_{L^2(dw)}^2 & \leq C_{\beta,M}s^{2|\beta| } \|(\lambda I-A_{s_1}^{(\varepsilon)})^{M}S^{(\varepsilon)}_tf\|_{L^2(dw)}^2\\
    &\quad +Cs^{2|\beta|+d(|\beta|)}\|S_t^{(\varepsilon)}f\|_{\mathcal H_{s_1}}^2.
\end{align*}
By Proposition~\ref{propo:b_0} and Theorem~\ref{teo:Lions} we have that $A^{(\varepsilon)}$ is the generator of the  semigroup $\{S^{(\varepsilon)}_t\}_{t \geq 0}$  of self-adjoint linear operators on $L^2(dw)$. Therefore, since $1/2 <t<2$, by the spectral theorem (or Cauchy integral formula) we obtain
\begin{align*}
   s^{2|\beta|} \|(\lambda I-A^{(\varepsilon)}_{s_1})^{M}S_t^{(\varepsilon)}f\|_{L^2(dw)}^2 & =s^{2|\beta|} \Big\|\Big(\lambda I-\frac{d}{dt}\Big)^{M}S_t^{(\varepsilon)}f\Big\|_{L^2(dw)}^2\\ & \leq Cs^{2|\beta|}\lambda^M \|f\|_{L^2(dw)}^2
  \leq C\exp(cs^{2\ell} ) \| f\|_{\mathcal H_{s_1}}^2 .
\end{align*}
Moreover, by Theorem~\ref{teo:lions_app} and the fact that $1/2<t<2$ we have
\begin{align*}
    s^{2|\beta|+d(|\beta|)}\|S_t^{(\varepsilon)}f\|_{\mathcal H_{s_1}}^2 \leq Cs^{2|\beta|+d(|\beta|)}\exp(c's^{2\ell})\|f\|_{\mathcal H_{s_1}}^2 \leq C' \exp(cs^{2\ell})\|f\|_{\mathcal H_{s_1}}^2,
\end{align*}
which completes the proof. 
\end{proof}
}
\subsection{Pointwise estimate for convolution kernels of semigroups}
{
\begin{coro}\label{coro:semigroup_pointwise}
Let $\varepsilon \in\{0,\varepsilon_0\}$. There are constants $C,c>0$ and $M \in \mathbb{N}$ such that for all $s>1/4$, $f \in \mathcal H_{2^{2M+1}s}$, $\mathbf{x} \in \mathbb{R}^N$, and $1/2<t<2$ we have
\begin{equation}
|S_t^{(\varepsilon)}f(\mathbf{x})|\leq C \exp(-s\|\mathbf{x}\|)\exp(cs^{2\ell})\|f\|_{\mathcal H_{2^{2M+1}s}}.
\end{equation}
\end{coro}
}
\begin{proof}
By~\eqref{eq:inverse}, the   Cauchy-Schwarz inequality,  and~\eqref{eq:Plancherel}, for $M \in \mathbb{N}$ such that $M> \mathbf{N}/2$ and for any function $g \in D(\Delta^{M})$, we have
\begin{equation}\label{eq:sobolev}
\begin{split}
|g(\mathbf{x})|&=c_k^{-1}\left|\int_{\mathbb{R}^N}E(i\xi,\mathbf{x})\mathcal{F}g(\xi)\,dw(\xi)\right|\\&= c_k^{-1}\left|\int_{\mathbb{R}^N}(1+\|\xi\|^2)^{-M}(1+\|\xi\|^2)^{M}E(i\xi,\mathbf{x})\mathcal{F}g(\xi)\,dw(\xi)\right|\\&\leq C_{M}\|(1+\|\xi\|^2)^{M}\mathcal{F}g\|_{L^2(dw)}  \\&= C_{M}'\|(I-\Delta)^{M}g\|_{L^2(dw)}.
\end{split}
\end{equation}
Therefore, if  for $f \in \mathcal H_{2^{2M+1}s}$ we plug  $g(\mathbf{x})=\eta (\mathbf{x}, s)S_t^{(\varepsilon)}f(\mathbf{x})$ in~\eqref{eq:sobolev} and use Lemma~\ref{lem:pointwise_pre}, we obtain the claim, because $\exp(s\|\mathbf x\|)\leq \eta(\mathbf x,s)$ for all $\mathbf{x} \in \mathbb{R}^N$ and $s>1/4$.
\end{proof}

\begin{lema}\label{coro:Roesler}
There is a constant $C>0$ such that for all $\mathbf{x},\xi \in \mathbb{R}^N$ we have
\begin{equation}
|E(i\xi,\mathbf{x})-1| \leq C\|\mathbf{x}\|\|\xi\|.
\end{equation}
\end{lema}

\begin{proof}
For all $\mathbf{x},\xi \in \mathbb{R}^N$ we have
\begin{align*}
E(i\xi,\mathbf{x})-1&=E(\xi,i\mathbf{x})-E(\xi,0)=\int_{0}^{1} \frac{d}{dt}E(\xi,it\mathbf{x})\,dt=i\int_{0}^{1}\langle \nabla_{\mathbf{x}} E(\xi,it\mathbf{x}),\mathbf{x}\rangle\, dt.
\end{align*}
Therefore, by Cauchy-Schwarz inequality and Lemma~\ref{lem:Roesler_Dunkl_kernel} we get
\begin{align*}
|E(i\xi,\mathbf{x})-1| \leq C \int_0^{1}\|\nabla_{\mathbf{x}} E(\xi,it\mathbf{x})\|\|\mathbf{x}\|\,dt \leq C \|\mathbf{x}\|\|\xi\|.
\end{align*}
\end{proof}

Recall that the kernel $q_t^{(\varepsilon)}(\mathbf x)$ is given by~\eqref{eq:q_tepsilon}. Our goal is to obtain pointwise estimates of  $q_t^{(\varepsilon)}$ for $t=1$.  \begin{lema}\label{lem:L_infty}
Let $\varepsilon \in \{0,\varepsilon_0\}$. There is a constant $C>0$ such that for all $\mathbf{x} \in \mathbb{R}^N$ we have
\begin{align*}
\|\tau_{\mathbf{x}}q_1^{(\varepsilon)}-q_1^{(\varepsilon)}\|_{L^{\infty}}\leq C\|\mathbf{x}\|.
\end{align*}
\end{lema}

\begin{proof}
For any $\mathbf{y} \in \mathbb{R}^N$ we have
\begin{align*}
\tau_{\mathbf{x}}q_1^{(\varepsilon)}(-\mathbf{y})-q_1^{(\varepsilon)}(-\mathbf{y})=c_k^{-1}\int_{\mathbb{R}^N}\mathcal{F}q_1^{(\varepsilon)}(\xi)E(i\xi,-\mathbf{y})[E(i\xi,\mathbf{x})-1]\,dw(\xi).
\end{align*}
Therefore, the claim is a consequence of Lemma~\ref{coro:Roesler}, Lemma~\ref{lem:Roesler_Dunkl_kernel}, and the fact that $q_1^{(\varepsilon)} \in \mathcal{S}(\mathbb{R}^N)$, so  $\mathcal{F}(q_1^{(\varepsilon)}) \in \mathcal{S}(\mathbb{R}^N)$ as well.
\end{proof}

Theorem~\ref{teo:main_no_translation} is a special case (for $\varepsilon=0$) of the theorem below.
\begin{teo}\label{teo:main_1}
Let $\varepsilon \in \{0,\varepsilon_0\}$. There are constants $C,c>0$ such that for all $\mathbf{x} \in \mathbb{R}^N$ we have
\begin{equation}\label{eq:estimate_q} 
|q_1^{(\varepsilon)}(\mathbf{x})|\leq C \exp(-c\|\mathbf{x}\|^{\frac{2\ell}{2\ell-1}}).
\end{equation}
\end{teo}

\begin{proof} Since $q_1^{(\varepsilon)}\in\mathcal S(\mathbb R^N)$, it suffices to prove \eqref{eq:estimate_q} for large $\| \mathbf x\|$.   For any $\mathbf{x} \in \mathbb{R}^N$, $s>1$ and $r>0$ we write 
\begin{align*}
q_1^{(\varepsilon)}(\mathbf{x})&=\frac{1}{w(B(0,r))}\int_{B(0,r)}q_1^{(\varepsilon)}(\mathbf{x})\,dw(\mathbf{y})=\frac{1}{w(B(0,r))}\int_{B(0,r)}[q_1^{(\varepsilon)}(\mathbf{x})-\tau_{-\mathbf{y}}q_1^{(\varepsilon)}(\mathbf{x})]\,dw(\mathbf{y})\\&+\frac{1}{w(B(0,r))}\int_{B(0,r)}\tau_{-\mathbf{y}}q_1^{(\varepsilon)}(\mathbf{x})\,dw(\mathbf{y})=J_1+J_2.
\end{align*}
By Lemma~\ref{lem:L_infty} we have
\begin{equation}\label{eq:main_1}
|J_1| \leq C\frac{1}{w(B(0,r))}\int_{B(0,r)}\|\mathbf{y}\|\,dw(\mathbf{y}) \leq Cr.
\end{equation}
Furthermore, it follows by the definition of the Dunkl translation that
\begin{align*}
    \int_{B(0,r)}\tau_{-\mathbf{y}}q_1^{(\varepsilon)}(\mathbf{x})\,dw(\mathbf{y})=\int_{B(0,r)}\tau_{\mathbf{x}}q_1^{(\varepsilon)}(-\mathbf{y})\,dw(\mathbf{y})=S_1^{(\varepsilon)}\chi_{B(0,r)}(\mathbf{x}).
\end{align*}
 Therefore, by Corollary~\ref{coro:semigroup_pointwise} and \eqref{eq:eta_comp}  we get that there is $M \in \mathbb{N}$ such that
\begin{equation}\label{eq:main_2}
\begin{split}
|J_2| &=w(B(0,r))^{-1}|S_1^{(\varepsilon)}\chi_{B(0,r)}(\mathbf{x})|\\&\leq Cw(B(0,r))^{-1}\exp(cs^{2\ell})\exp(-s\|\mathbf{x}\|)\|\chi_{B(0,r)}\|_{\mathcal H_{2^{2M+1}s}}. \\&\leq Cw(B(0,r))^{-1}\exp(cs^{2\ell})\exp(-s\|\mathbf{x}\|)w(B(0,r))^{1/2}\exp(2^{2M}sr) \\&\leq Cr^{-\frac{\mathbf{N}}{2}}\exp(c's^{2\ell})\exp(-s\|\mathbf{x}\|),
\end{split}
\end{equation}
where in the last inequality we have used ~\eqref{eq:behavior}. Therefore, taking into account~\eqref{eq:main_1} and~\eqref{eq:main_2} we obtain
\begin{equation}\label{eq:main_3}
|q_1^{(\varepsilon)}(\mathbf{x})| \leq \big(r+r^{-\frac{\mathbf{N}}{2}}\exp(c's^{2\ell})\exp(-s\|\mathbf{x}\|)\big).
\end{equation}
Set $$r=\big(\exp(c's^{2\ell})\exp(-s\|\mathbf{x}\|)\big)^{\frac{1}{\mathbf{N}/2+1}},$$
then~\eqref{eq:main_3} reduces to
\begin{equation}
|q_1^{(\varepsilon)}(\mathbf{x})| \leq C \big(\exp(c's^{2\ell})\exp(-s\|\mathbf{x}\|)\big)^{\frac{1}{\mathbf{N}/2+1}}.
\end{equation}
Finally, setting $s=\delta\|\mathbf{x}\|^{1/(2\ell-1)}$ for $\delta>0$ small enough we obtain the claim.
\end{proof}

\subsection{Pointwise estimations for the integral kernel of the semigroup}
The following proposition was proved in~\cite[Proposition 4.4]{DzHej2}.
\begin{propo}\label{propo:compact_supports}
There is a constant $C>0$ such that for any  $r_1,r_2>0$, any $f\in L^1(dw)$ such that $\supp f \subseteq B(0,r_2)$, any continuous radial function $\phi$ such that $\text{\rm supp}\, \phi \subseteq B(0,r_1)$, and for all $\mathbf{y} \in \mathbb{R}^N$ we have
\begin{align*}
\|\tau_{\mathbf{y}}(f * \phi)\|_{L^1(dw)} \leq C (r_1(r_1+r_2))^{\frac{\mathbf{N}}{2}}
\| \phi\|_{L^{\infty}}\|f\|_{L^1(dw)}.
\end{align*}
\end{propo}
The lemma below is a suitable adaptation of ~\cite[Proposition 4.10]{DzHej2}.
\begin{lema}\label{lem:fg_growth}
Let $a,b>1$ and $f,g$ be measurable functions such that $g$ is radial and continuous, and there are constants $C,c>0$ such that
\begin{equation}\label{eq:fg_growth}
    |f(\mathbf{x})| \leq C\exp(-c\|\mathbf{x}\|^{a}) \text{ and } |g(\mathbf{x})| \leq C\exp(-c\|\mathbf{x}\|^{b}) \text{ for all }\mathbf{x} \in \mathbb{R}^N.
\end{equation}
Then there are constants $C',c'>0$ such that for all $\mathbf{y} \in \mathbb{R}^N$ we have
\begin{align*}
    \int_{\mathbb{R}^N}|\tau_{\mathbf{y}}(f * g)(-\mathbf{x})|\exp(c'd(\mathbf{x},\mathbf{y})^{\min\{a,b\}})\,dw(\mathbf{x}) \leq C'.
\end{align*}
\end{lema}

\begin{proof}
Let $\Psi_0 \in C^{\infty}(-\frac{1}{2},\frac{1}{2})$ and $\Psi \in C^{\infty}(\frac{1}{8},1)$ be such that 
\begin{align*}
    1=\Psi_0(\|\mathbf{x}\|)+\sum_{n=1}^{\infty}\Psi(2^{-n}\|\mathbf{x}\|)=\sum_{n=0}^{\infty}\Psi_n(\|\mathbf{x}\|) \text{ for all }\mathbf{x} \neq 0.
\end{align*}
Set  $f_n(\mathbf{x})=f(\mathbf{x})\Psi_{n}(\|\mathbf{x}\|)$ and $g_j(\mathbf{x})=g(\mathbf{x})\Psi_j(\|\mathbf{x}\|)$, where $n,j \geq 0$. Clearly, $\tau_{\mathbf{y}}(f*g)=\sum_{n,j=0}^{\infty}\tau_{\mathbf{y}}(f_n*g_j)$ (see~\cite[Proposition 4.10]{DzHej2} for details). Since ${\rm supp}\,f_n \subseteq B(0,2^n)$ and ${\rm supp}\,g_j \subseteq B(0,2^j)$, we have 
$${\rm supp}\,f_n*g_j \subseteq B(0,2^j+2^n). $$ By Proposition~\ref{propo:compact_supports} we obtain
\begin{equation}\label{eq:fg_growth_comp}
\begin{split}
     &\int_{\mathbb{R}^N}|\tau_{\mathbf{y}}(f_n * g_j)(-\mathbf{x})|\exp(c'd(\mathbf{x},\mathbf{y})^{\min\{a,b\}})\,dw(\mathbf{x}) \\&\leq \exp(c'(2^j+2^n)^{\min\{a,b\}}) \int_{\mathbb{R}^N}|\tau_{\mathbf{y}}(f_n * g_j)(-\mathbf{x})|\,dw(\mathbf{x}) \\&\leq C_1 \exp(c'(2^j+2^n)^{\min\{a,b\}})2^{j \frac{\mathbf{N}}{2}}(2^j+2^n)^{\frac{\mathbf{N}}{2}}\|f_n\|_{L^1(dw)}\|g_j\|_{L^{\infty}}.
\end{split}
\end{equation}
By~\eqref{eq:fg_growth}  we have $\|f_n\|_{L^1(dw)} \leq C\exp(-2^{na}c/2)$ and $\|g_j\|_{L^{\infty}} \leq C\exp(-2^{jb}c)$, so~\eqref{eq:fg_growth_comp} leads to
\begin{align*}
  &\int_{\mathbb{R}^N}|\tau_{\mathbf{y}}(f * g)(-\mathbf{x})|\exp(c'd(\mathbf{x},\mathbf{y})^{\min\{a,b\}})\,dw(\mathbf{x}) \\&\leq  C_1 \sum_{n,j=0}^{\infty}
  \exp(c'(2^j+2^n)^{\min\{a,b\}})2^{j \frac{\mathbf{N}}{2}}(2^j+2^n)^{\frac{\mathbf{N}}{2}} \exp(-c2^{na-1}-c2^{jb}). 
\end{align*}
Finally, we see that if $c'>0$ is small enough, then the double series above is convergent, so we are done.
\end{proof}

\begin{proof}[Proof of Theorem~\ref{teo:theo2}]
We write
$$q_1=\mathcal{F}^{-1}(\mathcal{F}q_1^{(0)})=\mathcal{F}^{-1}((\mathcal{F}q_1^{(0)}e^{\varepsilon_0\|\cdot\|^{2}})e^{-\frac{\varepsilon_0}{2}\|\cdot\|^{2}}e^{-\frac{\varepsilon_0}{2}\|\cdot\|^{2}}))=q_1^{(\varepsilon_0)}*h_{\varepsilon_0/2}*h_{\varepsilon_0/2},$$
where $h_{\varepsilon_0/2}$ is the Dunkl heat kernel (see~\eqref{eq:heat_kernel}). This gives
\begin{align*}
    |q_1(\mathbf{x},\mathbf{y})|&=|\tau_{\mathbf{x}}((q_1^{(\varepsilon_0)}*h_{\varepsilon_0/2})*h_{\varepsilon_0/2})(-\mathbf{y})|\leq \int_{\mathbb{R}^N}|\tau_{-\mathbf{y}}(q_1^{(\varepsilon_0)}*h_{\varepsilon_0/2})(\mathbf{z})||h_{\varepsilon_0/2}(\mathbf{x},\mathbf{z})|\,dw(\mathbf{z})\\&\leq \int_{d(\mathbf{x},\mathbf{y})\leq 2d(\mathbf{x},\mathbf{z})}+\int_{d(\mathbf{x},\mathbf{y})\leq 2d(\mathbf{y},\mathbf{z})}=J_1+J_2.
\end{align*}
By Theorem~\ref{teo:heat} applied to $h_{\varepsilon_0/2}(\mathbf{x},\mathbf{z})$,  we have
\begin{align*}
    |J_1| &\leq C \int_{d(\mathbf{x},\mathbf{y})\leq 2d(\mathbf{x},\mathbf{z})} |\tau_{-\mathbf{y}}(q_1^{(\varepsilon_0)}*h_{\varepsilon_0/2})(\mathbf{z})|w(B(\mathbf{x},\varepsilon_0))^{-1}\exp(-c_{\varepsilon_0}d(\mathbf{x},\mathbf{z})^{2})\,dw(\mathbf{z})
    \\&\leq Cw(B(\mathbf{x},1))^{-1}\exp(-c'd(\mathbf{x},\mathbf{y})^{\frac{2\ell}{2\ell-1}})\int_{\mathbb{R}^N} |\tau_{-\mathbf{y}}(q_1^{(\varepsilon_0)}*h_{\varepsilon_0/2})(\mathbf{z})|\,dw(\mathbf{z}),
\end{align*}
where in the last inequality we have used the fact that the measure $dw$ is doubling (see~\eqref{eq:doubling}). The functions $f=q_1^{(\varepsilon_0)}$ and $g=h_{\varepsilon_0/2}$ satisfy the assumptions of Lemma~\ref{lem:fg_growth} with $a=\frac{2\ell}{2\ell-1}$ and $b=2$ respectively (see Theorems~\ref{teo:main_1} and~\ref{teo:heat}), so the last integral is bounded by a constant. 

Thanks to the inequality  $|h_{\varepsilon_0/2}(\mathbf{x},\mathbf{z})| \leq Cw(B(\mathbf{x},\varepsilon_0))^{-1}$ (see Theorem~\ref{teo:heat}), $|J_2|$ is less than
\begin{align*}
&Cw(B(\mathbf{x},\varepsilon_0))^{-1}\int\limits_{d(\mathbf{x},\mathbf{y})\leq 2d(\mathbf{y},\mathbf{z})} |\tau_{-\mathbf{y}}(q_1^{(\varepsilon_0)}*h_{\varepsilon_0/2})(\mathbf{z})|\exp(-cd(\mathbf{y},\mathbf{z})^{\frac{2\ell}{2\ell-1}})\exp(cd(\mathbf{y},\mathbf{z})^{\frac{2\ell}{2\ell-1}})\,dw(\mathbf{z})\\&\leq Cw(B(\mathbf{x},1))^{-1}\exp(-c'd(\mathbf{x},\mathbf{y})^{\frac{2\ell}{2\ell-1}})\int_{\mathbb{R}^N} |\tau_{-\mathbf{y}}(q_1^{(\varepsilon_0)}*h_{\varepsilon_0/2})(\mathbf{z})|\exp(cd(\mathbf{y},\mathbf{z})^{\frac{2\ell}{2\ell-1}})\,dw(\mathbf{z}).
\end{align*}
Since the functions $f=q_1^{(\varepsilon_0)}$ and $g=h_{\varepsilon_0/2}$ satisfy the assumptions of Lemma~\ref{lem:fg_growth} with $a=\frac{2\ell}{2\ell-1}$ and $b=2$ respectively, the last integral is bounded by a constant independent of $\mathbf y$, provided $c>0$ is small enough. The proof is complete. 
\end{proof}

\renewcommand{\thesubsection}{\Alph{subsection}}
\section{Appendix}

\subsection{Proof of Proposition~\ref{propo:definition_equivalence}}\label{sec:definition_equivalence}

\begin{lema}\label{lem:convolution_bounded}
Let $s>1/4$ and let $\Phi$ be a radial $C^{\infty}_c(\mathbb{R}^{N})$-function such that $\int \Phi \,dw=1$ and ${\rm supp}\,\Phi \subset B(0,1)$. There is a constant $C=C_{\Phi}>0$ such that for all $f \in \mathcal{H}_s$ we have
\begin{equation}
    \|\Phi_{1/n}*f\|_{\mathcal{H}_s} \leq C\|f\|_{\mathcal{H}_s}.
\end{equation}
Moreover, 
\begin{equation}\label{eq:approx}
  \lim_{n\to\infty}  \|f-\Phi_{1/n}*f\|_{\mathcal{H}_s} =0 \text{ for all }  f\in\mathcal H_s. 
\end{equation}
Here and subsequently, $\Phi_{1/n}(\mathbf{x})=n^{\mathbf{N}}\Phi(n\mathbf{x})$.
\end{lema}

\begin{proof}
Let us note that by the definition of $\eta(\mathbf{x},s)$ (see~\eqref{eq:eta}), there is a constant $C>0$ such that for all $\mathbf{x},\mathbf{y} \in \mathbb{R}^{N}$ and $s>1/4$ we have 
\begin{equation}\label{eq:eta_two_sides}
 e^{s\|\mathbf{x}\|} \leq \eta(\mathbf{x},s) \leq  C e^{s\|\mathbf{x}\|} \leq C e^{sd(\mathbf{x},\mathbf{y})+s\|\mathbf{y}\|},   
\end{equation}
therefore, by the Cauchy--Schwarz inequality,
\begin{equation}\label{eq:approx_CS}
\begin{split}
    &\|\Phi_{1/n}*f\|^2_{\mathcal{H}_s} \leq C\int_{\mathbb{R}^{N}} \Big|\int_{\mathbb{R}^N}\Phi_{1/n}(\mathbf{x},\mathbf{y})f(\mathbf{y})\,dw(\mathbf{y})\Big|^2e^{s\|\mathbf{x}\|}\,dw(\mathbf{x}) \\& \leq C\int_{\mathbb{R}^{N}} \int_{\mathbb{R}^N}|\Phi_{1/n}(\mathbf{x},\mathbf{y})|\,dw(\mathbf{y})\int_{\mathbb{R}^N}|\Phi_{1/n}(\mathbf{x},\mathbf{y})||f(\mathbf{y})|^2\,dw(\mathbf{y})e^{s\|\mathbf{x}\|}\,dw(\mathbf{x}) .
\end{split}
\end{equation}
Since $\Phi$ is radial, by~\eqref{eq:translation-bounded} (see also~\eqref{eq:change_var}) we have
\begin{equation}\label{eq:translation_bounded}
 \int_{\mathbb{R}^{N}}|\Phi_{1/n}(\mathbf{x},\mathbf{y})|\,dw(\mathbf{y}) \leq \int_{\mathbb{R}^{N}}|\Phi(\mathbf{y})|\,dw(\mathbf{y}) \leq C.
\end{equation}
Consequently, combining~\eqref{eq:eta_two_sides}  and~\eqref{eq:approx_CS} we get
\begin{equation}\label{eq:approx_consequence}
   \|\Phi_{1/n}*f\|^2_{\mathcal{H}_s} \leq C' \int_{\mathbb{R}^N}|f(\mathbf{y})|^2e^{s\|\mathbf{y}\|}\int_{\mathbb{R}^{N}}|\Phi_{1/n}(\mathbf{x},\mathbf{y})|e^{sd(\mathbf{x},\mathbf{y})}\,dw(\mathbf{x})\,dw(\mathbf{y}). 
\end{equation}
Because ${\rm supp}\,\Phi_{1/n} \subseteq B(0,1)$ for all $n \in \mathbb{N}$ and $\Phi_{1/n}$ is radial,~\eqref{eq:translation-radial} implies that ${\rm supp}\,\Phi_{1/n}(\cdot,\mathbf{y}) \subset \mathcal{O}(B(\mathbf{y},1))$ for all $\mathbf{y} \in \mathbb{R}^N$. Therefore, $d(\mathbf{x},\mathbf{y}) \leq 1$ for all $\mathbf{x} \in {\rm supp}\,\Phi_{1/n}(\cdot,\mathbf{y})$, so applying~\eqref{eq:translation_bounded} to~\eqref{eq:approx_consequence} we get
\begin{align*}
    \| \Phi_{1\slash n} * f \|^2_{\mathcal{H}_s} \leq C'e^{s}\int_{\mathbb{R}^{N}}|f(\mathbf{y})|^2e^{s\|\mathbf{y}\|}\,dw(\mathbf{y}) \leq C''e^s\|f\|_{\mathcal{H}_s}^2,
\end{align*}
where in the last inequality we have used~ the first inequality of~\eqref{eq:eta_two_sides}. 
 
To finish the proof it suffices to show that~\eqref{eq:approx} holds  for compactly supported $\mathcal H_s$-functions, because they form a dense set there. Fix $f\in \mathcal H_s$. Let $R>0$ be such that $\text{supp}\, f\subseteq B(0,R)$. Then ${\rm supp\,} f*\Phi_{1/n}\subset B(0,R+1)$. By~\eqref{eq:eta_comp} we get
 \begin{equation}
   \| f*\Phi_{1/n} -f\|^2_{\mathcal H_s} \leq e^{(R+1)s +1}\| f*\Phi_{1\slash n} -f\|^2_{L^2(dw)}. 
 \end{equation}
 The right-hand side of the above inequality tends to zero, since one can easily prove (using the Dunkl transform) that $\Phi_{1/n}$ is an approximate of the identity on $L^2(dw)$.  
\end{proof}

\begin{proof}[Proof of Proposition~\ref{propo:definition_equivalence}~\eqref{numitem:a_def}$\Rightarrow$\eqref{numitem:b_def}]
Let $\boldsymbol f=\{f_n\}_{n \in \mathbb{N}} \subset C^{\infty}_c(\mathbb{R}^{N})$  be a Cauchy sequence  in $V_{\ell,s}$. Clearly, by completeness of $\mathcal H_s$, there is $f\in\mathcal H_s\subset L^2(dw)$ such that $\lim_{n\to \infty} \| f_n-f\|_{\mathcal H_s}=0$.  Let $|\beta| \leq \ell$, by Corollary~\ref{coro:garding_pre} the sequence $\{T^{\beta}f_n\}_{n \in \mathbb{N}}$ is a Cauchy sequence in $\mathcal{H}_s$, thus  it converges to a function $f_{\beta,s}$ in $\mathcal H_s$ and in $L^2(dw)$ as well.   Let $\varphi \in C^{\infty}_c(\mathbb{R}^N)$. Integrating by parts we obtain 
\begin{equation}\begin{split}\label{eq:by_parts}
   (-1)^{|\beta|} \int_{\mathbb{R}^N} f(\mathbf x)T^\beta \varphi(\mathbf x)\, dw(\mathbf x)&  =\lim_{n\to \infty }
    (-1)^{|\beta|} \int_{\mathbb{R}^N} f_n(\mathbf x)T^\beta \varphi(\mathbf x)\, dw(\mathbf x)\\
    &=\lim_{n\to \infty}\int_{\mathbb{R}^N} T^{\beta}f_n(\mathbf x)\varphi(\mathbf x)\, dw(\mathbf x)\\
    &=\int_{\mathbb{R}^N} f_{\beta, s}(\mathbf x)\varphi(x)\, dw(\mathbf{x}). 
\end{split}\end{equation}
Assume now that ${\boldsymbol g}=\{ g_n\}_{n\in\mathbb N}$ is another Cauchy sequence in $V_{\ell, s}$, such that $\{g_n\}_{n \in \mathbb{N}}$ converge to the $f$ in $\mathcal H_s$. Then \eqref{eq:by_parts} implies that $g_{\beta,s}=f_{\beta,s}$ thus $\{g_n\}_{n\in\mathbb N}$ corresponds to the same element in $V_{\ell,s}$. Hence we have proved that for every element $\boldsymbol f$ in $V_{\ell,s}$  we can find a unique element  in $f\in \mathcal H_s$ which satisfies \eqref{eq:weak_der}.
\end{proof}

\begin{proof}[Proof of Proposition~\ref{propo:definition_equivalence}~\eqref{numitem:b_def}$\Rightarrow$\eqref{numitem:a_def}]
Let $\Phi$ be a radial $C^{\infty}_c(\mathbb{R}^{N})$-function such that $\int \Phi \,dw=1$ and ${\rm supp}\,\Phi \subset B(0,1)$.  Let $\Psi$ be a radial $C^\infty_c(\mathbb{R}^N)$-function such that $\Psi \equiv 1$ on $B(0,1)$ and $0 \leq \Psi \leq 1$. For $n \in \mathbb{N}$ we set
$$f_n(\mathbf{x})=\Psi(\mathbf{x}/n)\Phi_{1/n}*f(\mathbf{x}).$$
Since $f \in \mathcal{H}_s$, we have $f_n \in C^{\infty}_c(\mathbb{R}^{N})$ for all $n \in \mathbb{N}$. By iteration of~\eqref{eq:derivative_formula_1}, for all $\beta \in \mathbb{N}_0^{N}$ such that $|\beta| \leq \ell$, there are functions $\Psi_{\beta,\beta',\sigma} \in C^{\infty}_c(\mathbb{R}^N)$ such that 
\begin{align*}
    T^{\beta}f_n(\mathbf{x})&=T^{\beta}(\Phi_{1/n}*f)(\mathbf{x})\Psi(\mathbf{x}/n)\\&\quad +\sum_{\sigma \in G}\sum_{\beta' \in \mathbb{N}^{N}_0, \,|\beta'| <  |\beta|} n^{|\beta'|-|\beta|} T^{\beta'}(\Phi_{1/n}*f)(\sigma(\mathbf{x}))\Psi_{\beta,\beta',\sigma}
    (\mathbf{x}\slash n)
\end{align*}
(see also~\eqref{eq:derivative_iteration}). Therefore, by the definition of $f_{\beta',s}$, we get
\begin{equation}\label{eq:convolution_derivative}
\begin{split}
    T^{\beta}f_n(\mathbf{x})&=(\Phi_{1/n}*f_{\beta,s})(\mathbf{x})\Psi(\mathbf{x}/n)\\&\quad +\sum_{\sigma \in G}\sum_{\beta' \in \mathbb{N}^{N}_0, \,|\beta'| < |\beta|}n^{|\beta'|-|\beta|}(\Phi_{1/n}*f_{\beta',s})(\sigma(\mathbf{x}))\Psi_{\beta,\beta',\sigma}(\mathbf{x}\slash n).
\end{split}
\end{equation}
 It follows from~\eqref{eq:convolution_derivative} and Lemma~\ref{lem:convolution_bounded} that 
\begin{align*}
    \lim_{n \to \infty}\|T^{\beta}f_n-f_{\beta,s}\|_{\mathcal{H}_s}=0 \text{ for all }|\beta| \leq \ell, 
\end{align*}
which completes the proof of the proposition. 

\end{proof}

\subsection{Proof of Theorem \ref{teo:one_semigroup}.}\label{sec:selfadjoint}
We remark that Theorem \ref{teo:one_semigroup} is the   part~\eqref{numitem:lions_c} of Corollary~\ref{coro:inclusions}. 
{The operator $L^{(\varepsilon)}=(-1)^{\ell+1} \sum_{j=1}^m T_{\zeta j}^{2\ell}-\varepsilon \Delta$ is understood as a differential-difference operator acting on $C^\infty(\mathbb{R}^N)$-functions. We define its action on all $L^2(dw)$-functions by means of distributions, that is, 
\begin{equation}\label{eq:L-weak}
\int_{\mathbb R^N} (L^{(\varepsilon)}f)(\mathbf{x})\varphi(\mathbf{x})\, dw(\mathbf{x})=\int_{\mathbb R^N} f(\mathbf{x})(L^{(\varepsilon)}\varphi)(\mathbf{x})\, dw(\mathbf{x}) \text{ for all } \varphi\in C_c^\infty (\mathbb R^N).
\end{equation}

\begin{lema}\label{lem:characterization}
Let $f\in V_{\ell,s}$. Then $f\in D(A^{(\varepsilon)}_s)$ if and only if $L^{(\varepsilon)}f$ belongs to $\mathcal H_{s}$ in the  sense of distributions (cf. \eqref{eq:L-weak}). 
\end{lema}
\begin{proof}
Assume that $f\in D(A^{(\varepsilon)}_s)$. Set $g=A^{(\varepsilon)}_sf\in \mathcal H_s$. Fix $\varphi\in C_c^\infty(\mathbb R^N)$. We may assume that $\varphi$ is real-valued. Define $\psi(\mathbf x)=\varphi(\mathbf x)\eta(\mathbf x, s)^{-1}$. Then $\psi\in C_c^\infty(\mathbb R^N)\subset V_{\ell,s}$. By the definition of $A^{(\varepsilon)}_s$ (see Subsection~\ref{sec:semigroup}) we get 
\begin{equation*}
    \begin{split}
        \int_{\mathbb R^N} g(\mathbf x)\varphi (\mathbf x)\, dw(\mathbf x)
        &=
        \int_{\mathbb R^N} g(\mathbf x)\psi (\mathbf x)\eta(\mathbf x,s)\, dw(\mathbf x)=b_{s,\varepsilon}(f,\psi)\\
        &=-\int_{\mathbb R^N} \sum_{j=1}^m T_{\zeta_j}^\ell f(\mathbf x)T_{\zeta_j}^\ell (\psi(\mathbf x)\eta(\mathbf x,s))\, dw(\mathbf x)\\
        &\quad +\varepsilon \int_{\mathbb R^N} \sum_{j=1}^N T_{j}f(\mathbf x)T_{j}(\psi(\mathbf x)\eta(\mathbf x,s))\, dw(\mathbf x)\\
        &=\int_{\mathbb R^N} f(\mathbf x) L^{(\varepsilon)}\varphi(\mathbf x)\, dw(\mathbf x),
    \end{split}
\end{equation*}
which proves that $g=L^{(\varepsilon)}f$ in the weak sense. 

Converselly, assume that $f\in V_{\ell,s}$ is such that $L^{(\varepsilon)}f\in \mathcal H_s$ in the
weak sense. Set $g=L^{(\varepsilon)}f$. Take $\varphi\in C_c^\infty(\mathbb R^N)$. Then $\varphi(\mathbf x)\eta(\mathbf x,s)\in C_c^\infty(\mathbb{R}^N)$ and 
\begin{equation}\label{eq:L-weak2}
    \begin{split}
        \int_{\mathbb R^N } g(\mathbf x)(\varphi(\mathbf x)\eta(\mathbf x, s))\, dw(\mathbf x)
        &= \int_{\mathbb R^N } f(\mathbf x)L^{(\varepsilon)}(\varphi(\mathbf x)\eta(\mathbf x,s))\, dw(\mathbf x)\\
        &= b_{s,\varepsilon}(f,\varphi). 
    \end{split}
\end{equation}
By a density  argument (see Subsection \ref{weight_sobol}),    the formula ~\eqref{eq:L-weak2} holds for all $\varphi \in V_{\ell,s}$, which implies that $f\in D(A^{(\varepsilon)}_s)$ and $A^{(\varepsilon)}_sf=g$. 
\end{proof}
\begin{coro}\label{coro:inclusions}  Let $\varepsilon \in \{0,\varepsilon_0\}$ and $c_0$ be the constant from \eqref{eq:semigroup_growth}. For every $s_1>s_2>1/4$ we have 
\begin{enumerate}[(a)]
    \item{ $D(A^{(\varepsilon)}_{s_1})\subset D(A^{(\varepsilon)}_{s_2}) \subset D(A^{(\varepsilon)})$ and $ A^{(\varepsilon)}_{s_1}\subset A^{(\varepsilon)}_{s_2}\subset A^{(\varepsilon)}$;}\label{numitem:lions_a}
    \item{$R(\lambda; A^{(\varepsilon)}_{s_1})\subset R(\lambda; A^{(\varepsilon)}_{s_2})\subset R(\lambda; A^{(\varepsilon)})$ for all $\lambda>c_0s_1^{2\ell}$, where $R(\lambda; A^{(\varepsilon)}_{s})$ denotes the resolvent operator, that is, $R(\lambda; A^{(\varepsilon)}_{s_j})=(\lambda I-A^{(\varepsilon)}_{s_j})^{-1}$}\label{numitem:lions_b}
    \item{$S_t^{\{\varepsilon,s_1\}}\subset S_t^{\{\varepsilon,s_2\}}\subset S^{(\varepsilon)}_t$ for all $t>0$.}\label{numitem:lions_c}
\end{enumerate}
\end{coro}
\begin{proof} 
The statements ~\eqref{numitem:lions_a} and~\eqref{numitem:lions_b} are consequences of Lemma \ref{lem:characterization}. To prove~\eqref{numitem:lions_c} we take $\omega>0$ sufficiently large. Then, by the Lions theorem (see Theorem \ref{teo:Lions}),  the operators $\widetilde{ A^{(\varepsilon)}_{s_1}}=A^{(\varepsilon)}_{s_1}-\omega I$, $\widetilde{A^{(\varepsilon)}_{s_2}}=A^{(\varepsilon)}_{s_2}-\omega I$, and $\widetilde{ A^{(\varepsilon)}}=A^{(\varepsilon)}-\omega I$, generate contraction semigroups  $\{e^{-t\omega}S_t^{\{\varepsilon,s_1\}}\}_{t \geq 0}$,   $\{e^{-t\omega}S_t^{\{\varepsilon,s_2\}}\}_{t \geq 0}$, and  $\{e^{-t\omega}S^{(\varepsilon)}_t\}_{t \geq 0}$ respectively (each semigroup acts on its corresponding Hilbert space $\mathcal H_{s_j}$). It follows from the statements~\eqref{numitem:lions_a} and~\eqref{numitem:lions_b} that the Yosida approximations  of $\widetilde{A^{(\varepsilon)}_{s_j}}$ (see \cite[Section 3.1]{Pazy}) satisfy  
$$ \lambda^2 R(\lambda; \widetilde{A^{(\varepsilon)}_{s_1}})-\lambda I\subset \lambda^2 R(\lambda; \widetilde{ A^{(\varepsilon)}_{s_2}})-\lambda I\subset\lambda^2 R(\lambda, \widetilde{A^{(\varepsilon)}})-\lambda I,$$ 
for $\lambda>0$, which implies~\eqref{numitem:lions_c}, by the proof of the Hille-Yosida theorem (see \cite{Pazy}). 
\end{proof}
}
\subsection{Proof of Proposition  \ref{pro:core}}\label{sec:core}
Since $\lambda>c_0 s^{2\ell}$, the operator $\lambda I-A_{s}^{(\varepsilon)}$ is invertible on $\mathcal H_s$. Let $R(\lambda; A_s^{(\varepsilon)})$ denote its inverse. Since $R(\lambda; A^{(\varepsilon)}_s)^n$ is bounded operator on $\mathcal H_s$,  it suffices to prove that 
$(\lambda I-A^{(\varepsilon)}_s)^n (C_c^\infty(\mathbb R^N))$ is a dense subspace in $\mathcal H_s$. For this purpose let 
$$\mathcal V_s^\infty=\{ f\in C^\infty (\mathbb R^N): T^\beta f\in \mathcal H_s\quad  \text{\rm for every } \beta \in \mathbb N_0^N\}.$$ 
We claim  that $\mathcal V_s^\infty$ is a core for $(\lambda I-A^{(\varepsilon)}_s)^n$, because for $f\in C_c^\infty(\mathbb R^N)$ we have $T^\beta R(\lambda; A^{(\varepsilon)}_s)^n f=  R(\lambda; A^{(\varepsilon)}_s)^n T^\beta f \in D((A^{(\varepsilon)}_s)^n)\subset \mathcal H_s$ and, consequently, $R(\lambda; A^{(\varepsilon)}_s)^n f\in \mathcal V_s^\infty$. Therefore $C_c^\infty(\mathbb R^N)\subset (\lambda I-A^{(\varepsilon)}_s)^n (\mathcal V_s^\infty)$, which proves the claim. 

Let $\Psi$ be as in Appendix~\ref{sec:definition_equivalence} and let $f\in\mathcal V_s^\infty$. Then $f_j(\mathbf x)=\Psi(\mathbf x\slash j) f(\mathbf x) \in C_c^\infty(\mathbb R^N)$ for all $j \in \mathbb{N}$. It is not difficult to prove that $\lim_{j\to\infty} \| T^\beta f_j-T^\beta f\|_{\mathcal H_s}=0$ for every multi-index $\beta\in \mathbb N_0^N$, which finishes the proof of the proposition.

\end{document}